\documentclass[11pt,leqno]{article}

\usepackage[T1]{fontenc}
\usepackage{lmodern}
\usepackage[a4paper,textwidth=405bp,textheight=646bp,centering]{geometry}
\usepackage{amsmath,amssymb,amsthm,mathtools,mathrsfs,bm}
\usepackage{enumitem}
\usepackage{aliascnt}
\usepackage{microtype}
\usepackage{xcolor}
\usepackage[numbers,sort&compress]{natbib}
\usepackage[colorlinks=true,linkcolor=blue,citecolor=blue,urlcolor=blue]{hyperref}
\usepackage{indentfirst}
\usepackage[nameinlink,noabbrev]{cleveref}

\allowdisplaybreaks
\numberwithin{equation}{section}

\newtheorem{theorem}{Theorem}[section]

\newaliascnt{proposition}{theorem}
\newtheorem{proposition}[proposition]{Proposition}
\aliascntresetthe{proposition}

\newaliascnt{lemma}{theorem}
\newtheorem{lemma}[lemma]{Lemma}
\aliascntresetthe{lemma}

\newaliascnt{corollary}{theorem}

\aliascntresetthe{corollary}

\newaliascnt{definition}{theorem}
\newtheorem{definition}[definition]{Definition}
\aliascntresetthe{definition}

\newaliascnt{assumption}{theorem}
\newtheorem{assumption}[assumption]{Assumption}
\aliascntresetthe{assumption}

\newaliascnt{problem}{theorem}

\aliascntresetthe{problem}

\theoremstyle{remark}
\newaliascnt{remark}{theorem}
\newtheorem{remark}[remark]{Remark}
\aliascntresetthe{remark}

\crefname{theorem}{theorem}{theorems}
\crefname{proposition}{proposition}{propositions}
\crefname{lemma}{lemma}{lemmas}
\crefname{corollary}{corollary}{corollaries}
\crefname{definition}{definition}{definitions}
\crefname{assumption}{assumption}{assumptions}
\crefname{problem}{problem}{problems}
\crefname{remark}{remark}{remarks}

\newcommand{\R}{\mathbb R}
\newcommand{\N}{\mathbb N}
\newcommand{\Pp}{\mathbb P}
\newcommand{\Q}{\mathbb Q}
\newcommand{\E}{\mathbb E}
\newcommand{\F}{\mathcal F}

\newcommand{\cE}{\mathcal E}

\newcommand{\dd}{\,\mathrm d}
\newcommand{\one}{\mathbf 1}
\newcommand{\ip}[2]{\left\langle #1,#2\right\rangle}
\newcommand{\norm}[1]{\left\lVert #1\right\rVert}
\newcommand{\abs}[1]{\left\lvert #1\right\rvert}
\newcommand{\br}[1]{\left\langle #1\right\rangle}
\newcommand{\esssup}{\operatorname*{ess\,sup}}

\title{Backward stochastic partial differential equations with unbounded random coefficients
\thanks{ This work is supported by
		the National Key R\&D Program of China (2022YFA1006102),
		the National Natural Science Foundation of China (12471418, 12595294, 12231002),
		and the New Cornerstone Science Foundation (NCI202501).}}
\author{Jie Xiong
\thanks{Department of Mathematics and Shenzhen International Center for Mathematics,
		Southern University of Science and Technology, Shenzhen 518055, China.
		(\texttt{xiongj@sustech.edu.cn}).}
		 \and 
		 Wen Xu
\thanks{School of Mathematical Sciences, Peking University,
		Beijing 100871, China.
		(\texttt{xuwen@math.pku.edu.cn}).}
		\and
		Zuo Quan Xu
	\thanks{	Department of Applied Mathematics, The Hong Kong Polytechnic University, Kowloon, Hong Kong
		SAR, China. 	(\texttt{maxu@polyu.edu.hk})}
		 \and 
		 Ying Yang
\thanks{Department of Mathematics,
		Southern University of Science and Technology, Shenzhen 518055, China.
		(\texttt{12331007@sustech.edu.cn}).}}
\date{}
\hypersetup{
  pdftitle={Backward stochastic partial differential equations with unbounded random coefficients },
  pdfauthor={Jie Xiong, Wen Xu, Zuoquan Xu, Ying Yang},
  pdfsubject={Nonlinear filtering, dual BSPDEs, and measure-valued uniqueness},
  pdfkeywords={nonlinear filtering, Zakai equation, BSPDE, unbounded observation drift}
}

\begin{document}
\maketitle

\begin{abstract}
We study a linear backward stochastic partial differential equation 
with unbounded observation-adapted random coefficients, motivated
by the duality approach to nonlinear filtering.
The coefficients depend predictably on the observation history through a
control process. The drift and the coefficient multiplying the martingale integrand may grow
linearly in space.  The diffusion matrix may  grow quadratically in space
and is allowed to be degenerate. Under bounded positive-order
spatial derivatives coefficients and bounded smooth terminal data,
we prove existence and uniqueness in a polynomially weighted Sobolev
class up to control-energy stopping times. The proof combines entropy bounds for likelihoods, direct spatial differentiation of the parameterized BSDE with unbounded coefficient, conditional parameter estimates, and stochastic-flow reconstruction.
\end{abstract}



\medskip


\section{Introduction}
\label{sec:introduction}

Nonlinear filtering seeks to characterize the conditional distribution of a hidden signal given the available noisy observations. When both the signal and observation processes evolve according to stochastic differential equations, the normalized filter satisfies the Kushner--Stratonovich equation, while its unnormalized counterpart satisfies the linear Zakai equation, see,  \citet{Kallianpur1980,BainCrisan2009,Xiong2008}. Classical results establish existence and uniqueness for these measure-valued stochastic partial differential equations (SPDEs) under boundedness assumptions on the coefficients. A more challenging question is whether these equations still uniquely characterize the filter when the coefficients are random and unbounded.

Under a probability $\Pp$, the motivating partially observed diffusion model is
\begin{equation}
	\begin{aligned}
		\dd X_t&=b(X_t,u_t)\dd t+\sigma(X_t,u_t)\dd W_t,\\
		\dd Y_t&=h(X_t,u_t)\dd t+\dd V_t,
	\end{aligned}
	\label{eq:intro-model}
\end{equation}
where $X$ is the unobserved $\R^d$-valued state process,
$Y$ is the $\R^m$-valued observation process, and
$W$ and $V$ are independent Brownian motions under $\Pp$
of dimensions $r$ and $m$, respectively, and $u$ is predictable with respect to the
observation filtration. 

Allowing the coefficients to depend on an observation-predictable process \(u\) introduces an additional difficulty in the duality argument for filtering equations. Although \(u\) is adapted, a pathwise solution of the associated backward PDE generally depends on future values of \(u\), and hence need not be adapted to the observation filtration. As emphasized by \citet{CrisanPardoux2026}, this anticipation prevents a direct application of the classical PDE-based duality argument and motivates the use of an adapted backward stochastic partial differential equation (BSPDE) with an additional martingale term. In our setting, this difficulty is compounded by the unboundedness of the coefficients and the input \(u\), requiring a construction that preserves observation adaptation while providing sufficient spatial regularity and integrability.

A standard approach to nonlinear filtering is to introduce a reference probability measure under which the observation process $Y$ is a Brownian motion. To study uniqueness of the Zakai equation related to system \eqref{eq:intro-model}, one may use duality with a  BSPDE, leading to the equation considered in this paper. Our main objective is to establish existence and uniqueness for this BSPDE with random and unbounded coefficients, which provides a basis for studying the uniqueness of solutions to the corresponding Zakai equation.

Under the reference probability $\Q$, let $Y$ denote the Brownian driver of
the BSPDE and let $u$ be $\mathbb F^Y$-predictable with
$\int_0^T|u_s|^2\dd s<\infty$ almost surely. The process $u$ may depend on the entire history of $Y$ up to the current time and need not admit a representation of the Markovian form $u_t=U(t,Y_t)$.  We allow
\[
|b(x,v)|+\|\sigma(x,v)\|+|h(x,v)|\le C(1+|x|+|v|),
\]
so that $a=\sigma\sigma^\top$ can have quadratic spatial growth.
For a stopping time $\theta$ bounded by a control-energy stopping time, we study
\begin{equation}
	\begin{cases}
		\dd f_t(x)=-\bigl(L_tf_t(x)+h(x,u_t)^\top g_t(x)\bigr)\dd t
		+g_t(x)^\top\dd Y_t,\\
		f_\theta(x)=\gamma(x).
	\end{cases}
	\label{eq:intro-dual-bspde}
\end{equation}
Here $f$ is scalar, $g$ takes values in $\R^m$, and $L_t$ is the
diffusion generator defined in \eqref{eq:generator-definition}.

The study of BSPDEs is closely related to stochastic control, stochastic Feynman--Kac representations, and duality methods for filtering equations.  The first approach to studying BSPDEs is based on a priori estimates in Sobolev spaces. For example,
 Ma and Yong \cite{MaYong1999} established existence and uniqueness of adapted solutions, spatial regularity, and comparison results.  \citet{DuTangZhang2013}   established \(W^{m,p}\)-estimates in the degenerate setting. However, both works require the relevant coefficients 
  to be uniformly bounded. Consequently, their results do not cover the case of unbounded coefficients with linear growth. For further results on the well-posedness, maximum principles, and
 regularity of BSPDEs, we refer to
 \citet{QiuTang2012,DuZhang2013,TangWei2016} and the references therein.

A complementary approach relates BSPDEs to parameterized backward stochastic differential equations (BSDEs) through stochastic flows.  Under boundedness
and regularity assumptions, \citet{Tang2005} constructed adapted
solutions by composing random fields generated by parameterized
BSDEs with the inverse of a forward stochastic flow.
\citet{MaYinZhang2012} developed a non-Markovian decoupling-field
framework connecting FBSDEs with degenerate BSPDEs in weighted
Sobolev spaces. Although these works provide probabilistic constructions and
representation results, additional estimates are needed to
handle the case where both the forward coefficients and the
coefficient multiplying the martingale integrand are unbounded.  For further references on stochastic flows, we refer the reader to
\citet{ChenTang2015}, \citet{BallyMatoussi2001} and the references therein. Moreover, \citet{BayerQiuYao2022} establish existence and uniqueness of weak
solutions within a prescribed growth condition for a class of BSPDEs
arising in rough-volatility models. The present equation
allows spatially growing diffusion coefficients together with
the term $h(x,u_t)^\top g_t$, whose unbounded coefficient
$h(x,u_t)$ is not covered by their assumptions on the generator.

A key technique used in this paper is to relate the BSPDE
\eqref{eq:intro-dual-bspde} directly to the parameterized BSDE
\eqref{eq:untruncated-BSDE} through stochastic flows.
Entropy estimates make the likelihood a true martingale, and an
equivalent change of probability gives the bounded solution.
We construct and estimate its variational equations before identifying
them as spatial derivatives by difference quotients. This yields the
required estimates under the reference probability without a
coefficient truncation or a moment assumption on the likelihood density.
The inverse-flow reconstruction \eqref{eq:reconstruct-limit}
involves an auxiliary Brownian motion $W$. We establish observation
adaptation of $f$ through the diagonal representation
\eqref{eq:diagonal-Y} in the reduced filtration
\eqref{eq:future-filtration}. Filtration shrinkage and predictable
representation in the stopped observation filtration eliminate the
auxiliary $W$-martingale term. Quadratic covariations with $Y$ then
identify an observation-predictable version of $g$, giving the BSPDE
directly in Proposition~\ref{prop:characteristic-equation}.

The main contributions of this paper are as follows.
\begin{enumerate}
	\item
	We establish existence and uniqueness for the degenerate
	BSPDE \eqref{eq:intro-dual-bspde} when $b$, $\sigma$, and $h$
	may grow linearly in both the spatial variable and the
	observation-predictable input $u$.
	Under the stated spatial regularity assumptions, the result
	allows dependence on the entire observation history,
	observation-measurable random terminal fields
	$\gamma\in\mathscr G_\theta^{n_*+1}$, and random horizons
	$\theta\le\tau_k$.
	
	\item
	We resolve the measurability difficulty arising from the
	auxiliary Brownian motion $W$ in the stochastic-flow construction.
	For the reconstructed fields \eqref{eq:reconstruct-limit},
	we prove pointwise observation adaptation of a jointly continuous
	version of $f$ and joint observation predictability of $g$.
	The argument identifies and eliminates the auxiliary
	$W$-martingale term directly.
	
	\item
	We establish well-posedness and spatial derivative estimates
	for the parameterized BSDE \eqref{eq:untruncated-BSDE} with
	unbounded coefficient $h$.
	The variational solutions are constructed under equivalent
	probabilities and estimated under the reference measure $\Q$.
	Difference quotients under a fixed equivalent probability,
	together with higher reference-measure moments, identify these
	solutions as strong parameter derivatives.
	The argument does not assume exponential integrability of
	$\int_t^T|H_s^{t,x}|^2\,\mathrm{d}s$.
	
	\item
	We obtain quantitative spatial regularity and growth
	control for the reconstructed solution.
	Specifically, we establish the weighted Sobolev estimate
	\eqref{eq:weighted-Sobolev-estimate}, the weighted spatial
	supremum estimates \eqref{eq:weighted-sup-2} for derivatives
	through order two, and almost sure joint continuity
	of $f$ in time and space.
	These properties provide the regularity and growth
	control needed for subsequent duality arguments
	with measure-valued filtering equations in a companion article \cite{xu}.
\end{enumerate}

The paper is organized as follows.
\Cref{sec:assumptions} introduces the probabilistic framework
and assumptions and states the main theorem.
\Cref{sec:parameter-bsde} develops the stochastic-flow and
parameterized BSDE construction, proves well-posedness for the
BSDE with unbounded coefficient, and establishes its spatial
parameter derivatives and estimates directly.
\Cref{sec:bspde} reconstructs the BSPDE solution through
composition with the inverse stochastic flow and establishes
observation adaptation, weighted spatial estimates, and uniqueness.
Finally, \Cref{sec:conclusion} provides concluding remarks
and discusses the connection to nonlinear filtering.



\section{Assumptions and statement of the main theorem}
\label{sec:assumptions}

Fix $T>0$ and positive integers $d,r,m,p$. Let
$\Omega=C_0([0,T];\R^r)\times C_0([0,T];\R^m)$ with product Wiener
probability $\Q$, completed sigma-field $\mathcal F$, and coordinate
processes $(W,Y)$.  Under $\Q$, $W$ and $Y$ are independent
standard Brownian motions of dimensions $r$ and $m$,
respectively. Let $\mathbb F$ be their usual augmented natural
filtration and $\mathbb F^Y$ the usual augmented natural filtration of $Y$,
all completions use the same $\Q$-null sets. Thus $(W,Y)$ is a product
Brownian motion in $\mathbb F$, $\mathcal F_0$ is trivial modulo null sets,
and $\mathbb F$ and $\mathbb F^Y$ have the predictable representation
property with respect to $(W,Y)$ and $Y$, respectively. 

Let $U\subset\R^p$ be Borel, equipped with the Euclidean norm, and fix an
$\mathbb F^Y$-predictable $U$-valued process $u$ such that
$\int_0^T|u_s|^2\dd s<\infty$ almost surely. Define
\begin{equation}
	\tau_k:=\inf\left\{t\in[0,T]:\int_0^t|u_s|^2\dd s\ge k\right\}\wedge T,
	\qquad k\in\N,\quad \inf\varnothing=\infty,
	\label{eq:control-stopping}
\end{equation}
where $\N=\{1,2,\ldots\}$. These are $\mathbb F^Y$-stopping times,
$\int_0^{\tau_k}|u_s|^2\dd s\le k$ almost surely, and
$\tau_k\uparrow T$ almost surely. The results below are stated for a fixed
$k$ and an $\mathbb F^Y$-stopping time $0\le\theta\le\tau_k$, their
constants may depend on $k$.

Define 
\begin{equation}
n_*:=\left\lfloor\frac d2\right\rfloor+3.
\label{eq:n-star}
\end{equation}
This is the smallest integer satisfying
$
n_*>\frac d2+2.
$ 
 Write $\br{x}:=(1+|x|^2)^{1/2}$ and let $\mathrm{Id}$ denote the identity matrix.

For $\ell>0$ and integers $n\ge0$ and $q\ge1$, let
$H_\ell^n(\R^d;\R^q)$ consist of all measurable functions
$v=(v^1,\ldots,v^q):\R^d\to\R^q$ whose weak spatial
derivatives up to order $n$ are square-integrable with
respect to the weighted measure $\br{x}^{-2\ell}\,\dd x$.
This space is equipped with the norm
\begin{equation}
	\|v\|_{H_\ell^n(\R^d;\R^q)}^2
	:=\sum_{a=1}^q\sum_{|\alpha|\le n}
	\int_{\R^d}|D^\alpha v^a(x)|^2
	\br{x}^{-2\ell}\,\dd x.
	\label{eq:weighted-Hn}
\end{equation}
When $q=1$, we write $H_\ell^n(\R^d)$, and when $d=1$, we write  $H_\ell^n$. 
We omit the domain and target space when they are clear
from the context. For a deterministic $t\in[0,T]$ and $q\ge2$,
let $\mathbb S^q$ and $\mathbb H^q$ consist of continuous
$\mathbb F$-adapted and $\mathbb F$-predictable processes on
$[t,T]$, respectively, with finite norms
\[
 \norm{U}_{\mathbb S^q}^q
 :=\E^\Q\sup_{t\le s\le T}|U_s|^q,
 \qquad
 \norm{V}_{\mathbb H^q}^q
 :=\E^\Q\left(\int_t^T|V_s|^2\dd s\right)^{q/2}.
\]
The time interval and target dimension will be clear from the
context. A probability or filtration specified in parentheses
replaces $\Q$ or $\mathbb F$ in this definition.

For smooth $\varphi$, define
\begin{equation}
	L_t\varphi(x):=b(x,u_t)\cdot\nabla\varphi(x)
	+\frac12\sum_{i,j=1}^d a^{ij}(x,u_t)D_{ij}\varphi(x),
	\label{eq:generator-definition}
\end{equation}
where $	 a(x,u_t):=\sigma(x,u_t)\sigma(x,u_t)^\top.$
Its adjoint is
\[
 L_t^*\varphi=-\sum_iD_i(b^i\varphi)
 +\frac12\sum_{i,j}D_{ij}(a^{ij}\varphi),
\]
with coefficients evaluated at $(x,u_t)$.  

Throughout this paper, \(B_M\) (or   \(B_R\)), with an uppercase subscript \(M\) (or   \(R\)), denotes the ball of radius \(M\) (or   \(R\)) centred at the origin in \(\mathbb R^d\), to distinguish it from the process \(B_s\).  For $k=0,1,2$, we write
\[
\|f\|_{C^k(B_R)}
:=\sum_{|\alpha|\le k}\sup_{x\in B_R}|D^\alpha f(x)|,
\]
with $C^0(B_R)=C(B_R)$. For \(1\le p,q<\infty\), \(L^q(\mathbb Q)\) denotes the space of random variables with finite \(q\)-th moment, equipped with the norm \(\|X\|_{L^q(\mathbb Q)}=(\mathbb E^{\mathbb Q}|X|^q)^{1/q}\), while \(L^2(B_M)\) denotes the space of square-integrable functions on \(B_M\), with norm \(\|v\|_{L^2(B_M)}=(\int_{B_M}|v(y)|^2\,dy)^{1/2}\). We write \(L^p(\mathbb Q\otimes dy;B_M)\) for the space of jointly measurable random fields \(h(\omega,y)\) with finite norm \(\|h\|_{L^p(\mathbb Q\otimes dy;B_M)}=(\int_{B_M}\mathbb E^{\mathbb Q}|h(y)|^p\,dy)^{1/p}\). 

For \(M>0\) and integers \(n\ge0\) and \(q\ge1\), let \(H^n(B_M;\mathbb R^q)\) denote the Sobolev space of \(\mathbb R^q\)-valued functions whose weak derivatives of orders up to \(n\) are square-integrable on \(B_M\), equipped with the norm
\[
\|v\|_{H^n(B_M;\mathbb R^q)}^2
:=
\sum_{a=1}^q\sum_{|\alpha|\le n}
\int_{B_M}|D^\alpha v^a(x)|^2\,\dd x. 
\]
In the scalar case, we write \(H^n(B_M)\).


We impose the following assumptions. 

\begin{assumption}
\label{ass:filtering}
The functions
\[
 b:\R^d\times U\to\R^d,
 \qquad
 \sigma:\R^d\times U\to\R^{d\times r},
 \qquad
 h:\R^d\times U\to\R^m
\]
are Borel measurable.  There is a constant $K\ge1$ such that, for all $x,y\in\R^d$ and $v\in U$,
\begin{align}
 &\abs{b(x,v)-b(y,v)}+\norm{\sigma(x,v)-\sigma(y,v)}
   \le K\abs{x-y},
   \label{eq:filtering-lipschitz}\\
 &\abs{b(x,v)}+\norm{\sigma(x,v)}+\abs{h(x,v)}
   \le K(1+\abs{x}+\abs{v}).
   \label{eq:filtering-growth}
\end{align}
\end{assumption}

\begin{assumption}
\label{ass:bspde-regularity}
For every multi-index
$\alpha$ with $1\le|\alpha|\le n_*+1$, the positive-order spatial derivatives
$
 D_x^\alpha b, D_x^\alpha\sigma, D_x^\alpha h
$ 
exist and are bounded uniformly on $\R^d\times U$. 
\end{assumption}

For an integer $n\ge0$, let $\mathscr G_\theta^n$ denote
the class of real-valued,
$\mathcal F^Y_\theta\otimes\mathcal B(\R^d)$-measurable
terminal fields $\gamma$ that are $C^n$ in $x$,
$\Q$-a.s., and satisfy
\begin{equation}
	\|\gamma\|_{n,\infty}
	:=\operatorname*{ess\,sup}_{\omega}
	\sum_{|\alpha|\le n}\sup_{x\in\R^d}
	|D_x^\alpha\gamma(\omega,x)|<\infty.
	\label{eq:terminal-class}
\end{equation}


We now define a solution to the BSPDE \eqref{eq:intro-dual-bspde}.

\begin{definition}
\label{def:bspde-solution}
Let $\theta\le\tau_k$ be an $\mathbb F^Y$-stopping time and
let $\gamma\in\mathscr G_\theta^{n_*+1}$.
We call a pair $(f,g)$ a weighted Sobolev solution
of the BSPDE \eqref{eq:intro-dual-bspde} if, for some $\ell>0$,
the fields admit jointly measurable representatives such that,
for every deterministic $s\in[0,T]$, the map
$(\omega,x)\mapsto f_{s\wedge\theta}(\omega,x)$ is
$\mathcal F_s^Y\otimes\mathcal B(\R^d)$-measurable, and the
extension of $g$ by zero after $\theta$ is
$\mathcal P^Y\otimes\mathcal B(\R^d)$-measurable, where
$\mathcal P^Y$ is the predictable sigma-field of $\mathbb F^Y$
on $\Omega\times[0,T]$. In particular, for these representatives,
$f_{\cdot\wedge\theta}(x)$ is adapted and $g_\cdot(x)$ is
predictable for every fixed $x\in\R^d$.
We also require $f\in L^\infty(\Omega\times[0,\theta]\times\R^d)$ and
\begin{align}
\E^\Q\left[
	\esssup_{0\le s\le\theta}\norm{f_s}_{H_\ell^{n_*}}^2
	+\int_0^\theta
	\norm{g_s}_{H_\ell^{n_*}(\R^d;\R^m)}^2\,\dd s
	\right]<\infty.
	\label{eq:solution-norm}
\end{align}
Moreover, for every $\varphi\in C_c^\infty(\R^d)$
and $t\in[0,T]$, the identity
\begin{align}
	\ip{f_{t\wedge\theta}}{\varphi}
	={}&\ip{\gamma}{\varphi}
	+\int_{t\wedge\theta}^{\theta}
	\ip{f_s}{L_s^*\varphi}\,\dd s
	\notag\\
	&+\int_{t\wedge\theta}^{\theta}
	\ip{h(\cdot,u_s)^\top g_s}{\varphi}\,\dd s
	-\int_{t\wedge\theta}^{\theta}
	\ip{g_s}{\varphi}^\top\,\dd Y_s
	\label{eq:weak-bspde}
\end{align}
holds $\Q$-a.s..
\end{definition}
We now present the main theorem of this article.

\begin{theorem}
		\label{thm:main-bspde}
	Suppose Assumptions \ref{ass:filtering} and
	\ref{ass:bspde-regularity} hold.
	Fix $k\in\N$, an $\mathbb F^Y$-stopping time
	$\theta\le\tau_k$, and $\gamma\in\mathscr G_\theta^{n_*+1}$.
	Then \eqref{eq:intro-dual-bspde} admits a unique
	weighted Sobolev solution $(f,g)$ with bounded first
	component, satisfying \eqref{eq:solution-norm}
	with $\ell=d/2+2$.
	Moreover, $f$ admits a jointly measurable version that is
	$\Q$-almost surely continuous in $(t,x)$ on
	$[0,\theta]\times\R^d$.
	For every $\lambda\ge d/2+2$ and every multi-index
	$\alpha$ with $|\alpha|\le2$, the solution satisfies
	\begin{align}
		\E^\Q\Bigg[
		&\esssup_{0\le s\le\theta}\sup_{x\in\R^d}
		\frac{|D^\alpha f_s(x)|^2}{\br{x}^{2\lambda}}
		+\int_0^\theta\sup_{x\in\R^d}
		\frac{|D^\alpha g_s(x)|^2}{\br{x}^{2\lambda}}\dd s
		\Bigg]<\infty.
		\label{eq:weighted-sup-2}
	\end{align}
	\end{theorem}

Note the left of \eqref{eq:weighted-sup-2} is decreasing in $\lambda$, so we have 
	\begin{align}
		\sum_{|\alpha|\le 2}\E^\Q\Bigg[
		&\esssup_{0\le s\le\theta}\sup_{x\in\R^d}
		\frac{|D^\alpha f_s(x)|^2}{\br{x}^{d+4}}
		+\int_0^\theta\sup_{x\in\R^d}
		\frac{|D^\alpha g_s(x)|^2}{\br{x}^{d+4}}\dd s
		\Bigg]<\infty.
	\end{align}


\begin{remark}
	The derivatives $D^\alpha f$ and $D^\alpha g$ in
	\eqref{eq:weighted-sup-2} are understood as weak spatial
	derivatives. Since $n_*>d/2+2$, local Sobolev embedding gives
	$C^2$  representatives of $f_s$ and $g_s$.
	Their weak derivatives up to order two coincide with the
	classical derivatives of these representatives. Therefore, 
	the spatial suprema in \eqref{eq:weighted-sup-2} are taken
	over the resulting continuous representatives of the derivatives.
\end{remark}

\section{Parameterized BSDEs}
\label{sec:parameter-bsde}

In this section, we study the parameterized BSDE associated
with the stopped forward stochastic flow \eqref{eq:forward-flow}.
We first construct the likelihood and the bounded solution of
\eqref{eq:untruncated-BSDE}. We then construct variational solutions,
prove their conditional and reference-measure estimates, and identify
them as spatial derivatives by difference quotients. In particular,
\eqref{eq:uniform-F-parameter} and \eqref{eq:uniform-G-parameter}
hold directly for the BSDE with unbounded coefficient.

\subsection{Forward stochastic flow}

For $0\le t\le T$ and $x\in\R^d$,  define 
\begin{equation*}
 A_t:=\{t\le\theta\},\quad
 \chi_s^t:=\one_{\{t<s\le\theta\}},\quad 0\le s\le T.
\end{equation*}
The process $\chi^t$ is predictable. Define the forward flow $X^{t,x}$ on the deterministic interval $[0,T]$ by
\begin{equation}
	X_s^{t,x}
	=x+\int_0^s\chi_a^t b(X_a^{t,x},u_a)\dd a
	+\int_0^s\chi_a^t\sigma(X_a^{t,x},u_a)\dd W_a.
	\label{eq:forward-flow}
\end{equation}
Write
$
 \Phi_{t,s}(x):=X_s^{t,x} $ and $
 \Phi_s:=\Phi_{0,s}. 
$
The indicator $\chi^t$ stops the dynamics at $\theta$.
The drift and diffusion coefficients remain potentially unbounded,
with linear growth in the spatial and control variables.

The following Lemma \ref{98} follows from \citet[Theorem~4.6.5]{Kunita1990}.

\begin{lemma}\label{98}
		Under Assumptions \ref{ass:filtering} and \ref{ass:bspde-regularity},
	$(\Phi_{t,s})_{0\le t\le s\le T}$ is a stochastic flow of
	$C^{n_*}$-diffeomorphisms.
	\end{lemma}
Define 
\[
\Psi_{t,s}:=\Phi_{t,s}^{-1},
\quad
J_{t,s}(x):=\det D_x\Phi_{t,s}(x)>0,
\quad
\Psi_s:=\Psi_{0,s},
\quad
J_s:=J_{0,s}.
\]

The following Lemma \ref{lem:flow-estimates} gives the regularity and moment estimates for the stopped flow.

\begin{lemma}[Flow estimates]
	\label{lem:flow-estimates}
	Under Assumptions \ref{ass:filtering} and \ref{ass:bspde-regularity}.
	Let $\widetilde{\Q}$ be any probability measure equivalent to $\Q$
	under which $W$ remains an $\mathbb F$-Brownian motion.
Then,	for every $2\le q<\infty$ and $j=1,\ldots,n_*$,
	\begin{align}
		\E^{\widetilde{\Q}}
		\sup_{t\le s\le T}|\Phi_{t,s}(x)|^q
		&\le C_{k,q}\br{x}^q,
		\label{eq:flow-state-moment}\\
		\sup_{0\le t\le T}\sup_{x\in\R^d}
		\E^{\widetilde{\Q}}
		\sup_{t\le s\le T}|D_x^j\Phi_{t,s}(x)|^q
		&\le C_{j,q}.
		\label{eq:flow-derivative-moment}
	\end{align}
	Moreover,
	\begin{align}
		\sup_{0\le t\le T}\sup_{x\in\R^d}
		\E^{\widetilde{\Q}}\sup_{t\le s\le T}
		\Bigg[
		&J_{t,s}(x)^q+J_{t,s}(x)^{-q}
		+\sum_{j=1}^{n_*}
		\big|(D^j\Psi_{t,s})(\Phi_{t,s}(x))\big|^q
		\Bigg]
		\le C_{k,q}.
		\label{eq:maximal-inverse-flow-moments}
	\end{align}
\end{lemma}

\begin{proof}
	 
First, it is easy to verify:
	\begin{align*}
		&\E^{\widetilde{\Q}}
		\sup_{t\le a\le s}|\Phi_{t,a}(x)|^q\\
		&\quad\le C_q\bigl(1+|x|^q+k^{q/2}\bigr)
		+C_q\int_t^s
		\E^{\widetilde{\Q}}
		\sup_{t\le r\le a}|\Phi_{t,r}(x)|^q\dd a.
	\end{align*}
Then,	Gronwall's inequality yields
	\[
	\E^{\widetilde{\Q}}
	\sup_{t\le s\le T}|\Phi_{t,s}(x)|^q
	\le C_q\bigl(1+|x|^q+k^{q/2}\bigr)
	\le C_{k,q}\br{x}^q.
	\]
This	proves \eqref{eq:flow-state-moment}.
	
Denote
	\[
	B_s:=\chi_s^tD_xb(\Phi_{t,s}(x),u_s),
	\quad
	C_s^\nu:=\chi_s^tD_x\sigma^\nu(\Phi_{t,s}(x),u_s),
	\]
where $\sigma^\nu$  denotes the $\nu$th column of the diffusion matrix  $\sigma$. 
Then, for a positive constant $C$,  we have,  
	\[
	|B_s|+\sum_{\nu=1}^r|C_s^\nu|\le C.
	\]
	For $1\le j\le n_*$,  we have 
\begin{align*}
	\dd D_x^j\Phi_{t,s}(x)
	={}&\bigl[B_sD_x^j\Phi_{t,s}(x)+R_s^{(j)}\bigr]\dd s\\
	&+\sum_{\nu=1}^r
	\bigl[C_s^\nu D_x^j\Phi_{t,s}(x)
	+S_s^{(j),\nu}\bigr]\dd W_s^\nu,
\end{align*}
with
\[
D_x\Phi_{t,t}(x)=\mathrm{Id},
\qquad
D_x^j\Phi_{t,t}(x)=0\quad(j\ge2).
\]
Here, the remainders $R_s^{(j)}$ and $S_s^{(j),\nu}$ collect
the chain-rule terms involving only spatial derivatives of
the flow of orders below $j$.
Note that $R^{(1)}_s=S^{(1),\nu}_s=0$, and for $j\ge2$,
	\begin{equation}
		|R_s^{(j)}|+|S_s^{(j)}|
		\le C_j\sum_{\ell=2}^j
		\ \sum_{\substack{i_1+\cdots+i_\ell=j\\i_a\ge1}}
		\prod_{a=1}^{\ell}|D_x^{i_a}\Phi_{t,s}(x)|,
		\label{eq:flow-derivative-forcing}
	\end{equation}
	
For $j= 1$, \eqref{eq:flow-derivative-moment} is clearly true, we proceed by induction on $j$. 	Suppose the estimate holds for every order below $j$ and every
	finite exponent $q\ge2$. For each product in
	\eqref{eq:flow-derivative-forcing}, all $i_a<j$,  H\"older's
	inequality gives
	\begin{align*}
		&\E^{\widetilde{\Q}}\sup_{t\le s\le T}
		\prod_{a=1}^{\ell}|D_x^{i_a}\Phi_{t,s}(x)|^q\\
		&\quad\le
		\prod_{a=1}^{\ell}
		\left(
		\E^{\widetilde{\Q}}\sup_{t\le s\le T}
		|D_x^{i_a}\Phi_{t,s}(x)|^{q\ell}
		\right)^{1/\ell}
		\le C_{j,q}.
	\end{align*}
	Consequently,
	\[
	\E^{\widetilde{\Q}}\int_t^T
	\bigl(|R_s^{(j)}|^q+|S_s^{(j)}|^q\bigr)\dd s
	\le C_{j,q}.
	\]
Then, BDG and Gronwall inequalities yield \eqref{eq:flow-derivative-moment}, which
completes the induction.
	
Now, we set
	\[
	\mathcal J_s:=D_x\Phi_{t,s}(x),
	\qquad
	\mathcal K_s:=\mathcal J_s^{-1}.
	\]
Then the	Matrix It\^o calculus yields
	\[
	\dd\mathcal K_s
	=\mathcal K_s\left(-B_s+\sum_{\nu=1}^rC_s^\nu C_s^\nu\right)\dd s
	-\sum_{\nu=1}^r\mathcal K_sC_s^\nu\dd W_s^\nu,
	\qquad \mathcal K_t=\mathrm{Id}.
	\]
Since	the coefficients are uniformly bounded, we have
	\begin{equation}\label{x2}
	\sup_{0\le t\le T}\sup_{x\in\R^d}
	\E^{\widetilde{\Q}}\sup_{t\le s\le T}|\mathcal K_s|^q
	\le C_q
	\qquad(q\ge2).
	\end{equation}
	For $j\ge2$, we have 
	\begin{align*}
		&\big|(D^j\Psi_{t,s})(\Phi_{t,s}(x))\big|\\
		&\quad\le C_j|\mathcal K_s|^j
		\sum_{\ell=1}^{j-1}
		\sum_{\substack{i_1+\cdots+i_\ell=j\\i_a\ge1}}
		\big|(D^\ell\Psi_{t,s})(\Phi_{t,s}(x))\big|
		\prod_{a=1}^{\ell}|D_x^{i_a}\Phi_{t,s}(x)|,
	\end{align*}
by induction which implies 
	\begin{equation}\label{x1}
	\sup_{0\le t\le T}\sup_{x\in\R^d}
	\E^{\widetilde{\Q}}\sup_{t\le s\le T}
	\big|(D^j\Psi_{t,s})(\Phi_{t,s}(x))\big|^q
	\le C_{j,q},
	\qquad 1\le j\le n_*.
	\end{equation}
	
Note that, $J_{t,s}(x)=\det\mathcal J_s>0$ and
	$J_{t,s}(x)^{-1}=\det\mathcal K_s$. Then the Hadamard's inequality (see \citet{Lange2014}) gives
	\[
	J_{t,s}(x)^q+J_{t,s}(x)^{-q}
	\le C_{d,q}\left(
	|\mathcal J_s|^{dq}+|\mathcal K_s|^{dq}
	\right).
	\]
Finally, combine \eqref{x2}, \eqref{x1} and \eqref{eq:flow-derivative-moment} prove
	\eqref{eq:maximal-inverse-flow-moments}.
\end{proof}
	
\begin{remark}
The formulation under $\widetilde{\Q}$ allows the flow
estimates to be used in the likelihood argument in
Proposition~\ref{prop:param-likelihood} and in the uniform
BSDE derivative estimates in
Proposition~\ref{prop:uniform-parameter}.  The changes of measure involved there
	are induced by stochastic exponentials driven by $Y$,
	and hence preserve the $\mathbb F$-Brownian property of $W$.
\end{remark}

\subsection{Parameterized likelihood martingales}

For $(t,x)$, and $0\le s\le T$,  define 
\begin{equation}
 H_s^{t,x}:=\chi_s^t h(X_s^{t,x},u_s), 
 \label{eq:Htx}
\end{equation}
and define
\begin{equation}
 Z_s^{t,x}
 :=\cE\left(\int_0^{s}(H_a^{t,x})^\top\dd Y_a\right).
 \label{eq:Ztx}
\end{equation}

We next justify the change of measure induced by \(Z^{t,x}\). The proof is inspired by the localization and entropy argument of Klebaner and Liptser \cite{Klebaner}, we include the details to establish the required estimates in our setting with random controls satisfying the stopped energy bound \(\int_0^\theta |u_s|^2\,\mathrm ds\le k\).

\begin{proposition}
\label{prop:param-likelihood}
For every $(t,x)\in[0,T]\times\R^d$, $Z^{t,x}$ is a uniformly
integrable $\Q$-martingale on $[0,T]$. Consequently,
\begin{equation}
	\frac{\dd\widehat\Q^{t,x}}{\dd\Q}\bigg|_{\F_T}
	=Z_T^{t,x}=Z_\theta^{t,x}
	\label{eq:Qhat-tx}
\end{equation}
defines an equivalent probability measure.
Under $\widehat\Q^{t,x}$, $W$ and the process
\begin{equation}
	\widehat Y_s^{t,x}
	:=Y_s-\int_0^s H_a^{t,x} \dd a,
	\qquad 0\le s\le T,
	\label{eq:Yhat-tx}
\end{equation}
are independent $\mathbb F$-Brownian motions.
\end{proposition}

\begin{proof}
Let
\[
 \rho_j^{t,x}
 :=\inf\left\{s\in[0,T]:\int_0^s|H_a^{t,x}|^2\dd a\ge j\right\}\wedge T.
\]
The exponential \eqref{eq:Ztx} stopped at $\rho_j^{t,x}$ is a true martingale.  Under the
probability $\widehat\Q_j^{t,x}$ on $\F_T$ with density
$Z_{\rho_j^{t,x}}^{t,x}$, the process
\[
 Y_s-\int_0^{s\wedge\rho_j^{t,x}}H_a^{t,x}\dd a,
 \qquad 0\le s\le T,
\]
is an Brownian motion independent of $W$.   Applying Lemma 
\ref{lem:flow-estimates}  gives
\[
 \sup_j\E^{\widehat\Q_j^{t,x}}
     \sup_{t\le s\le T}|X_s^{t,x}|^2
 \le C_k\br{x}^2.
\]
Therefore, 
\begin{equation}
 \sup_j\E^{\widehat\Q_j^{t,x}}
     \int_0^{\rho_j^{t,x}}|H_s^{t,x}|^2\dd s
 \le C_k\br{x}^2.
 \label{eq:param-H-entropy-bound}
\end{equation}
Applying It\^o's formula to $\log Z_{\rho_j^{t,x}}^{t,x}$, 
then, we have
\begin{equation}
 \E^\Q\left[Z_{\rho_j^{t,x}}^{t,x}
             \log Z_{\rho_j^{t,x}}^{t,x}\right]
 =\frac12\E^{\widehat\Q_j^{t,x}}
      \int_0^{\rho_j^{t,x}}|H_s^{t,x}|^2\dd s
 \le C_k\br{x}^2.
 \label{eq:param-entropy}
\end{equation}

 By \eqref{eq:param-entropy}
and $r\log^+r\le r\log r+e^{-1}$ with $r >0$, where
$\log^+r:=\max\{\log r,0\}$, for every $R>1$, and as $R\to\infty$, we have 
\[
\sup_j\E^\Q\!\left[Z_{\rho_j^{t,x}}^{t,x}\one_{\{Z_{\rho_j^{t,x}}^{t,x}>R\}}\right]
\le \frac{1}{\log R}\sup_j\E^\Q[Z_{\rho_j^{t,x}}^{t,x}\log^+Z_{\rho_j^{t,x}}^{t,x}]
\le \frac{C_k\br{x}^2+e^{-1}}{\log R}
\longrightarrow0.
\]

Thus, $\{Z_{\rho_j^{t,x}}^{t,x}\}_{j\ge1}$ is uniformly integrable, and $Z_{\rho_j^{t,x}}^{t,x}\to Z_T^{t,x}$ a.s., which implies that  $Z^{t,x}$ is a uniformly
integrable $\Q$-martingale on $[0,T]$.
Therefore,
\eqref{eq:Qhat-tx} defines an equivalent probability measure.
Note that, $Z_T^{t,x}=Z_\theta^{t,x}$ because $H^{t,x}=0$ after $\theta$. Then, 
Girsanov's theorem applied to $(W,Y)$  shows that
$W$ and $\widehat Y^{t,x}$ are independent $\mathbb F$-Brownian
motions under $\widehat\Q^{t,x}$. This completes the proof. 

\end{proof}

\begin{remark}
	The stopping factor $\chi^t$ in \eqref{eq:Htx} is crucial
	for the martingale argument under the present assumptions.
	Since $\theta\le\tau_k$,
	$
	\int_0^T\chi_s^t|u_s|^2\,\mathrm{d}s\le k,
 \Q\text{-a.s.}
	$
	This bound remains valid under the localized equivalent
	measures and, together with Lemma~\ref{lem:flow-estimates},
	yields \eqref{eq:param-H-entropy-bound} and the uniform
	entropy estimate \eqref{eq:param-entropy}.
	Without this stopping, almost sure finiteness of
	$\int_0^T|u_s|^2\,\mathrm{d}s$ alone does not guarantee
	that the stochastic exponential is a true martingale.
\end{remark}

\subsection{The parameterized BSDE with unbounded coefficient}

For deterministic $(t,x)\in[0,T]\times\R^d$, let
$H^{t,x}$ be defined by \eqref{eq:Htx} and define
$
\Gamma^{t,x}:=\gamma(X_\theta^{t,x}).
$

We work  with the BSDE:
\begin{equation}
\left\{
\begin{aligned}
\dd F_s^{t,x}
&=-(H_s^{t,x})^\top G_s^{t,x}\dd s
+(G_s^{t,x})^\top\dd Y_s
+(\widetilde G_s^{t,x})^\top\dd W_s,
&&t\le s\le T,\\
F_T^{t,x}&=\Gamma^{t,x}.
\end{aligned}
\right.
\label{eq:untruncated-BSDE}
\end{equation}
Although $H^{t,x}$ may be unbounded, the likelihood martingale
in Proposition~\ref{prop:param-likelihood} removes its contribution
to the drift under an equivalent probability measure. We shall not
assume any exponential moment of $\int_t^T|H_s^{t,x}|^2\dd s$.

We first record an energy estimate whose use does not presuppose
integrability of the martingale integrands.
\begin{lemma}[Localized energy estimate]
\label{lem:localized-BSDE-energy}
Let $\tilde {\mathbb{Q}}$ be a probability measure under which $(B^1,B^2)$ is a
Brownian motion. Suppose that $U$ is continuous and
\[
\dd U_s=-(a_s^\top V_s+R_s)\dd s
+V_s^\top\dd B_s^1+\widetilde V_s^\top\dd B_s^2,
\]
where the integrands are locally square-integrable and the drift
is locally integrable. Then, for $q\ge2$ and a stopping time
$t\le\rho\le T$,
\begin{align}
&\E^{\tilde {\mathbb{Q}}}\left[
\left(\int_\rho^T(|V_s|^2+|\widetilde V_s|^2)\dd s\right)^{q/2}
\middle|\F_\rho\right]\notag\\
&\quad\le C_q\E^\mu\left[
\sup_{\rho\le s\le T}|U_s|^q
\left(1+\int_\rho^T|a_s|^2\dd s\right)^{q/2}
+\left(\int_\rho^T|R_s|\dd s\right)^q
\middle|\F_\rho\right],
\label{eq:localized-BSDE-energy}
\end{align}
provided the right-hand side is finite.
\end{lemma}
\begin{proof}
Stop all terms when their quadratic variations and variation exceed
a deterministic level. It\^o's formula for $|U|^2$, followed by
Cauchy--Schwarz and Young's inequalities, controls the stopped
quadratic variation by
\[
C\sup_{\rho\le s\le T}|U_s|^2
\left(1+\int_\rho^T|a_s|^2\dd s\right)
+C\sup_{\rho\le s\le T}|U_s|\int_\rho^T|R_s|\dd s
\]
and a stochastic integral. Conditional BDG and Young's inequalities
at exponent $q/2$ absorb the contribution of that stochastic
integral into the left-hand side. The resulting bound is uniform
in the localization level. Conditional Fatou's lemma gives
\eqref{eq:localized-BSDE-energy}.
\end{proof}

\begin{proposition}[Well-posedness and representation]
\label{prop:density-stability}
For every deterministic $(t,x)$, \eqref{eq:untruncated-BSDE} has a
unique solution with bounded first component and integrands in
$\mathbb H^2(\R^m)\times\mathbb H^2(\R^r)$. It satisfies
\begin{equation}
F_s^{t,x}
=\frac{\E^\Q[Z_T^{t,x}\Gamma^{t,x}\mid\F_s]}{Z_s^{t,x}}
=\E^{\widehat\Q^{t,x}}[\Gamma^{t,x}\mid\F_s],
\qquad |F_s^{t,x}|\le\|\gamma\|_\infty,
\label{eq:F-representation}
\end{equation}
and, for every finite $q\ge2$,
\begin{equation}
\E^\Q\left(\int_t^T
(|G_s^{t,x}|^2+|\widetilde G_s^{t,x}|^2)\dd s\right)^{q/2}
\le C_{k,q}\br{x}^q.
\label{eq:zero-order-G-moment}
\end{equation}
On $A_t$, the solution is constant after $\theta$, with zero
martingale integrands. On $A_t^c$,
\begin{equation}
\one_{A_t^c}F_s^{t,x}=\one_{A_t^c}\gamma(x),\qquad
\one_{A_t^c}G_s^{t,x}
=\one_{A_t^c}\widetilde G_s^{t,x}=0,
\qquad t\le s\le T.
\label{eq:BSDE-off-active-event}
\end{equation}
\end{proposition}
\begin{proof}
Fix $(t,x)$ and suppress these indices. By
Proposition~\ref{prop:param-likelihood}, $Z$ is strictly positive
and uniformly integrable. The continuous uniformly integrable
martingale
\[
M_s:=\E^\Q[Z_T\Gamma\mid\F_s],\qquad t\le s\le T,
\]
has a Brownian representation with predictable locally
square-integrable integrands $K^Y,K^W$. Define
\begin{equation}
F_s:=M_s/Z_s,\qquad
G_s:=K_s^Y/Z_s-F_sH_s,\qquad
\widetilde G_s:=K_s^W/Z_s.
\label{eq:limit-integrands-definition}
\end{equation}
Bayes' formula gives \eqref{eq:F-representation}, and It\^o's
formula for $M/Z$ gives \eqref{eq:untruncated-BSDE} with locally
square-integrable integrands. The growth assumption, the stopped
control-energy bound, and Lemma~\ref{lem:flow-estimates} give, for
every finite $p\ge1$,
\begin{equation}
\E^\Q\left(1+\int_t^T|H_s|^2\dd s\right)^p
\le C_{k,p}\br{x}^{2p}.
\label{eq:H-square-integrable}
\end{equation}
Apply Lemma~\ref{lem:localized-BSDE-energy} with $a=H$, $R=0$,
and use the bound on $F$. This proves
\eqref{eq:zero-order-G-moment}, in particular membership of the
integrands in $\mathbb H^2$.

For uniqueness, the difference of two solutions is a bounded
local martingale under $\widehat\Q^{t,x}$, driven by
$(\widehat Y^{t,x},W)$ and having terminal value zero. It is a
true martingale and thus vanishes identically. Its quadratic
variation then implies equality of the integrands, first
$\widehat\Q^{t,x}$-almost surely and hence also $\Q$-almost surely.
The same conditional representation and quadratic-variation
argument prove the assertions after $\theta$ and on $A_t^c$.
\end{proof}

\subsection{Spatial parameter regularity and estimates}

The next two conditional estimates will also be used when the
initial spatial parameter is varied while the probability measure
is kept fixed.
\begin{lemma}
\label{lem:conditional-linear-BSDE}
Let $B$ be Brownian under $\mu$ and have the martingale
representation property in the given filtration. If $q\ge2$,
$\eta\in L^q(\mu)$, and $\int_t^T|R_s|\dd s\in L^q(\mu)$,
the martingale representation of
$\E^\mu[\eta+\int_t^T R_s\dd s\mid\F_s]$ defines the solution
of
\[
\dd U_s=-R_s\dd s+Z_s\dd B_s,\qquad U_T=\eta.
\]
For every stopping time $t\le\rho\le T$,
\begin{align}
|U_\rho|^q
+\E^\mu\left[
\left(\int_\rho^T|Z_s|^2\dd s\right)^{q/2}
\middle|\F_\rho\right]
&\le C_q\E^\mu\left[
|\eta|^q+\left(\int_\rho^T|R_s|\dd s\right)^q
\middle|\F_\rho\right].
\label{eq:conditional-linear-BSDE}
\end{align}
The corresponding $\mathbb S^q\times\mathbb H^q$ stability
estimate holds for differences of terminal data and sources.
\end{lemma}
\begin{proof}
Conditional Jensen, Doob, and BDG inequalities applied to the
representing martingale give these estimates.
\end{proof}

\begin{lemma}
\label{lem:conditional-flow-polynomial}
Let $\mu$ be equivalent to $\Q$ and suppose that $W$ remains
Brownian under $\mu$. For $1\le j\le n_*$ and finite $q\ge2$,
there is a deterministic polynomial $P_{j,q}$ with nonnegative
coefficients such that
\begin{equation}
\E^\mu\left[
\sup_{\rho\le s\le T}\sum_{i=1}^j|D_x^iX_s^{t,x}|^q
\middle|\F_\rho\right]
\le P_{j,q}\left(1+\sum_{i=1}^j|D_x^iX_\rho^{t,x}|\right)
\label{eq:conditional-flow-polynomial}
\end{equation}
for every stopping time $t\le\rho\le T$. Moreover,
\begin{equation}
\E^\mu\left[
\sup_{\rho\le s\le T}|X_s^{t,z}-X_s^{t,x}|^q
\middle|\F_\rho\right]
\le C_q|X_\rho^{t,z}-X_\rho^{t,x}|^q.
\label{eq:conditional-flow-difference}
\end{equation}
The constants and polynomial coefficients depend only on the
coefficient derivative bounds, dimensions, $j,q,T$.
\end{lemma}
\begin{proof}
Apply the derivative equations used in
Lemma~\ref{lem:flow-estimates} on $[\rho,T]$, with their
$\F_\rho$-measurable initial values. Conditional BDG, Gronwall,
and induction on the derivative order give
\eqref{eq:conditional-flow-polynomial}. The same estimates for
the difference of two flow equations give
\eqref{eq:conditional-flow-difference}.
\end{proof}

\begin{proposition}[Strong parameter derivatives]
\label{prop:uniform-parameter}
For fixed $t$ and every finite $q\ge2$, the map
\[
x\longmapsto(F^{t,x},G^{t,x},\widetilde G^{t,x})
\]
is $n_*$ times continuously differentiable with values in
$\mathbb S^q\times\mathbb H^q(\R^m)\times\mathbb H^q(\R^r)$.
For $|\beta|\le n_*$,
\begin{align}
\E^\Q\sup_{t\le s\le T}|D_x^\beta F_s^{t,x}|^q
&\le C_{k,\beta,q},
\label{eq:uniform-F-parameter}\\
\E^\Q\left(\int_t^T
(|D_x^\beta G_s^{t,x}|^2
+|D_x^\beta\widetilde G_s^{t,x}|^2)\dd s\right)^{q/2}
&\le C_{k,\beta,q}\br{x}^q.
\label{eq:uniform-G-parameter}
\end{align}
The constants are independent of $t,x$. For fixed $t$, the
solution fields and their derivatives admit jointly measurable
versions in $(\omega,s,x)$; the integrand fields can be chosen
measurable with respect to the predictable sigma-field of
$\mathbb F$ times $\mathcal B(\R^d)$. These fields have the
stated derivatives as weak spatial derivatives and belong
locally to $H^{n_*}$ in space: for every $M<\infty$,
\begin{equation}
\E^\Q\left[
\sup_{t\le s\le T}\|F_s^{t,\cdot}\|_{H^{n_*}(B_M)}^2
+\int_t^T\left(\|G_s^{t,\cdot}\|_{H^{n_*}(B_M)}^2
+\|\widetilde G_s^{t,\cdot}\|_{H^{n_*}(B_M)}^2\right)\dd s
\right]<\infty.
\label{eq:local-parameter-Sobolev}
\end{equation}
Moreover, $F^{t,\cdot}$ has a version that is almost surely
continuous in $(s,x)$ on $[t,T]\times\R^d$, and
$|F_s^{t,x}|\le\|\gamma\|_\infty$ simultaneously for all
$(s,x)$ outside one null set.
\end{proposition}

\begin{proof}
Fix $t$ and write $X^x=X^{t,x}$, $H^x=H^{t,x}$,
$\Gamma^x=\Gamma^{t,x}$, and $\mu_x=\widehat\Q^{t,x}$.
We first construct solutions of the variational equations and
estimate them independently; only afterwards do we identify
them as parameter derivatives. Throughout the proof, a
polynomial $P$ has deterministic nonnegative coefficients
depending only on the displayed derivative orders and moment
exponents and on the fixed coefficient bounds.

The chain rule, bounded positive-order derivatives of $h$ and
$\gamma$, and the flow estimates imply, for $1\le|\eta|\le n_*$
and finite $p\ge1$,
\begin{equation}
\E^\mu\left(\int_t^T|D_x^\eta H_s^x|^2\dd s\right)^p
+\E^\mu|D_x^\eta\Gamma^x|^{2p}
\le C_{k,\eta,p}
\label{eq:BSDE-data-derivative-moments}
\end{equation}
under every equivalent measure $\mu$ preserving the Brownian
property of $W$. Conditional versions are bounded by a
polynomial in
\[
\mathfrak D_j^x(\rho):=1+\sum_{i=1}^j|D_x^iX_\rho^x|.
\]
For $H$ and $\Gamma$ themselves, the corresponding conditional
derivative statements below use only positive-order derivatives;
no moment of the likelihood density is involved.

\emph{Construction and conditional estimates.}
Set $R^0=0$ and
\[
 (U^0(x),V^0(x),\widetilde V^0(x))
 =(F^{t,x},G^{t,x},\widetilde G^{t,x}).
\]
Under $\mu_x$, $U^0(x)$
is a bounded martingale, so its integrands have all finite
$\mathbb H^p(\mu_x)$ moments. Suppose candidates of order
less than $j$ have been constructed. For $|\beta|=j$, put
\begin{equation}
\Gamma^\beta(x):=D_x^\beta\Gamma^x,\qquad
R_s^\beta(x):=\sum_{0<\eta\le\beta}\binom\beta\eta
(D_x^\eta H_s^x)^\top V_s^{\beta-\eta}(x).
\label{eq:R-beta}
\end{equation}
Here $\eta\le\beta$ is componentwise and
$\binom\beta\eta=\prod_i\binom{\beta_i}{\eta_i}$.
The lower-order candidates and
\eqref{eq:BSDE-data-derivative-moments} show, by
Cauchy--Schwarz at higher moment exponents, that
$\Gamma^\beta(x)$ and $\int_t^T|R_s^\beta(x)|\dd s$ have
every finite moment under $\mu_x$. The Brownian martingale
representation of
\[
\E^{\mu_x}\left[
\Gamma^\beta(x)+\int_t^T R_s^\beta(x)\dd s\middle|\F_s\right]
\]
therefore defines $(U^\beta,V^\beta,\widetilde V^\beta)$
uniquely in every $\mathbb S^p(\mu_x)\times\mathbb H^p(\mu_x)
\times\mathbb H^p(\mu_x)$, with
\begin{equation}
\dd U_s^\beta(x)=-R_s^\beta(x)\dd s
+V_s^\beta(x)^\top\dd\widehat Y_s^{t,x}
+\widetilde V_s^\beta(x)^\top\dd W_s,
\qquad U_T^\beta(x)=\Gamma^\beta(x).
\label{eq:derivative-under-Qhat}
\end{equation}

Lemma~\ref{lem:conditional-linear-BSDE} and induction yield
\begin{equation}
\begin{aligned}
|U_\rho^\beta(x)|^p
&+\E^{\mu_x}\left[
\left(\int_\rho^T(|V_s^\beta(x)|^2
+|\widetilde V_s^\beta(x)|^2)\dd s\right)^{p/2}
\middle|\F_\rho\right]\\
&\le P_{\beta,p}(\mathfrak D_{|\beta|}^x(\rho)).
\end{aligned}
\label{eq:closed-conditional-induction}
\end{equation}
For $\beta=0$ this is the bounded-martingale estimate, with a
constant polynomial. For the induction step, conditional
Cauchy--Schwarz gives
\begin{align*}
&\E^{\mu_x}\left[\left(\int_\rho^T|R_s^\beta(x)|\dd s\right)^p
\middle|\F_\rho\right]\\
&\quad\le C_{\beta,p}\sum_{0<\eta\le\beta}
\left\{\E^{\mu_x}\left[
\left(\int_\rho^T|D_x^\eta H_s^x|^2\dd s\right)^p
\middle|\F_\rho\right]\right\}^{1/2}\\
&\qquad\qquad\times
\left\{\E^{\mu_x}\left[
\left(\int_\rho^T|V_s^{\beta-\eta}(x)|^2\dd s\right)^p
\middle|\F_\rho\right]\right\}^{1/2}
\le P_{\beta,p}(\mathfrak D_{|\beta|}^x(\rho)).
\end{align*}
The last inequality uses the lower-order estimate at exponent
$2p$. This closes the induction without assuming the existence
of any parameter derivative of the BSDE solution.

\emph{Estimates under the reference probability.}
The bound for $U_\rho^\beta$ in
\eqref{eq:closed-conditional-induction} is a pathwise inequality.
Apply it at a countable dense set of times, use equivalence of
$\mu_x$ and $\Q$, and then use continuity. Lemma~\ref{lem:flow-estimates}
gives
\begin{equation}
\E^\Q\sup_{t\le s\le T}|U_s^\beta(x)|^p\le C_{k,\beta,p}.
\label{eq:candidate-F-moment}
\end{equation}
Under $\Q$, \eqref{eq:derivative-under-Qhat} is
\begin{equation}
\dd U_s^\beta(x)
=-[(H_s^x)^\top V_s^\beta(x)+R_s^\beta(x)]\dd s
+V_s^\beta(x)^\top\dd Y_s
+\widetilde V_s^\beta(x)^\top\dd W_s.
\label{eq:derivative-BSDE}
\end{equation}
Apply Lemma~\ref{lem:localized-BSDE-energy}. By
\eqref{eq:H-square-integrable} and
\eqref{eq:candidate-F-moment} at exponent $2p$,
\[
\E^\Q\left[
\sup_s|U_s^\beta(x)|^p
\left(1+\int_t^T|H_s^x|^2\dd s\right)^{p/2}\right]
\le C_{k,\beta,p}\br{x}^p.
\]
Induction on $|\beta|$, using
\eqref{eq:R-beta} and \eqref{eq:BSDE-data-derivative-moments}
at doubled exponents, gives
\[
\E^\Q\left(\int_t^T|R_s^\beta(x)|\dd s\right)^p
\le C_{k,\beta,p}\br{x}^p.
\]
Thus
\begin{equation}
\E^\Q\left(\int_t^T(|V_s^\beta(x)|^2
+|\widetilde V_s^\beta(x)|^2)\dd s\right)^{p/2}
\le C_{k,\beta,p}\br{x}^p.
\label{eq:candidate-G-moment}
\end{equation}
The case $\beta=0$ is already provided by
Proposition~\ref{prop:density-stability}. These are estimates for
the independently constructed candidates, not assumptions on
the derivatives to be identified.

\emph{Neighboring parameters under one fixed measure.}
Fix a base point $x$ and keep $\mu_x$ fixed while varying $z$.
Since $\dd\widehat Y^{t,z}
=\dd\widehat Y^{t,x}-(H^z-H^x)\dd s$, the candidate at $z$
satisfies under $\mu_x$
\begin{equation}
\begin{aligned}
\dd U_s^\beta(z)
={}&-[(H_s^z-H_s^x)^\top V_s^\beta(z)+R_s^\beta(z)]\dd s\\
&+V_s^\beta(z)^\top\dd\widehat Y_s^{t,x}
+\widetilde V_s^\beta(z)^\top\dd W_s.
\end{aligned}
\label{eq:fixed-measure-candidate}
\end{equation}
All integrands are initially locally square-integrable under
$\mu_x$ by equivalence. The pathwise $U^\beta(z)$ bound from
\eqref{eq:closed-conditional-induction} is still valid under
$\mu_x$. Moreover,
\[
|H_s^z-H_s^x|\le C|X_s^z-X_s^x|.
\]
Apply the conditional energy estimate to
\eqref{eq:fixed-measure-candidate}. The conditional flow
estimates and lower-order induction then give
\begin{equation}
\begin{aligned}
&\E^{\mu_x}\left[
\left(\int_\rho^T(|V_s^\beta(z)|^2
+|\widetilde V_s^\beta(z)|^2)\dd s\right)^{p/2}
\middle|\F_\rho\right]\\
&\qquad\le
P_{\beta,p}\bigl(\mathfrak D_{|\beta|}^z(\rho)
+|X_\rho^z-X_\rho^x|\bigr).
\end{aligned}
\label{eq:fixed-measure-candidate-moments}
\end{equation}
In detail, the first term on the right of the energy estimate
is controlled using conditional H\"older at doubled exponents
and \eqref{eq:conditional-flow-difference}; the source term is
controlled by the same conditional Cauchy--Schwarz calculation
as above, now using the lower-order instances of
\eqref{eq:fixed-measure-candidate-moments}. At $\rho=t$ the flow
derivatives have deterministic initial values, so these moments
are uniform for $z$ in a fixed neighborhood of $x$.

Comparison of \eqref{eq:fixed-measure-candidate} at $z$ and $x$
leaves the source
\[
(H^z-H^x)^\top V^\beta(z)+R^\beta(z)-R^\beta(x)
\]
and the terminal difference $\Gamma^\beta(z)-\Gamma^\beta(x)$.
The flow and data derivatives are continuous in their
parameter in every finite moment norm under $\mu_x$, by the
flow estimates, bounded coefficient derivatives, and dominated
convergence. Induction on $|\beta|$, together with
\eqref{eq:fixed-measure-candidate-moments}, shows that the
source tends to zero in $L^p(\mu_x;L^1(t,T))$.
Lemma~\ref{lem:conditional-linear-BSDE} consequently proves
continuity of all candidates in
$\mathbb S^p(\mu_x)\times\mathbb H^p(\mu_x)\times\mathbb H^p(\mu_x)$
as $z\to x$. Equivalence of measures gives convergence in
$\Q$-probability of these process norms. The locally uniform
higher moments \eqref{eq:candidate-F-moment}--\eqref{eq:candidate-G-moment}
and Vitali's theorem give continuity in the corresponding
$\Q$-norms as well. No comparison of likelihood density
moments is required.

\emph{Identification by difference quotients.}
We proceed by induction on $j\le n_*$. Suppose the candidates
of orders at most $j-1$ have been identified as continuous
strong derivatives. For $|\alpha|=j-1$, $y=x+\varepsilon e_i$,
and $\Delta_\varepsilon A=(A(y)-A(x))/\varepsilon$,
subtraction of the candidate equations under the fixed measure
$\mu_x$ gives
\begin{equation}
\begin{aligned}
\dd\Delta_\varepsilon U_s^\alpha
={}&-S_{\varepsilon,s}\dd s
+(\Delta_\varepsilon V_s^\alpha)^\top\dd\widehat Y_s^{t,x}
+(\Delta_\varepsilon\widetilde V_s^\alpha)^\top\dd W_s,\\
\Delta_\varepsilon U_T^\alpha
={}&\Delta_\varepsilon\Gamma^\alpha,
\end{aligned}
\label{eq:variational-difference-quotient}
\end{equation}
where, with an empty sum when $\alpha=0$,
\begin{align}
S_\varepsilon
={}&(\Delta_\varepsilon H)^\top V^\alpha(y)\notag\\*
&+\sum_{0<\eta\le\alpha}\binom\alpha\eta
\left[(\Delta_\varepsilon D_x^\eta H)^\top V^{\alpha-\eta}(y)
+(D_x^\eta H^x)^\top\Delta_\varepsilon V^{\alpha-\eta}\right].
\label{eq:difference-quotient-source}
\end{align}
Every difference quotient of a BSDE integrand occurring in
this source has lower order. The induction hypothesis and the
fundamental theorem of calculus in the process spaces give
\begin{equation}
\Delta_\varepsilon V^{\alpha-\eta}
=\int_0^1 V^{\alpha-\eta+e_i}(x+\lambda\varepsilon e_i)\dd\lambda,
\qquad 0<\eta\le\alpha.
\label{eq:lower-order-FTC}
\end{equation}
There are analogous parameter integral formulas for the data
derivatives. The identity \eqref{eq:lower-order-FTC}, established
under $\Q$, remains valid under $\mu_x$ by equivalence; the
integral is integrable there by
\eqref{eq:fixed-measure-candidate-moments}.

Candidate continuity under $\mu_x$, the data estimates, and
\eqref{eq:lower-order-FTC} imply
\[
S_\varepsilon\longrightarrow R^{\alpha+e_i}(x)
\quad\hbox{in }L^p(\mu_x;L^1(t,T)),\qquad
\Delta_\varepsilon\Gamma^\alpha
\longrightarrow\Gamma^{\alpha+e_i}(x)
\quad\hbox{in }L^p(\mu_x).
\]
Here the limit source is
$(D_iH^x)^\top V^\alpha(x)+D_iR^\alpha(x)
=R^{\alpha+e_i}(x)$, with the derivative of $R^\alpha$ involving
only the already identified lower-order candidates.
The binomial identity for \eqref{eq:R-beta} gives the last
equality. Stability therefore yields
\begin{equation}
\Delta_\varepsilon(U^\alpha,V^\alpha,\widetilde V^\alpha)
\longrightarrow
(U^{\alpha+e_i},V^{\alpha+e_i},\widetilde V^{\alpha+e_i})(x)
\label{eq:difference-quotient-Qhat-convergence}
\end{equation}
in all finite $\mathbb S^p(\mu_x)\times\mathbb H^p(\mu_x)
\times\mathbb H^p(\mu_x)$ norms.

For completeness, we establish uniform higher moments of
these difference quotients under $\Q$ before transferring the
strong convergence. Lemma~\ref{lem:conditional-linear-BSDE}
applied to \eqref{eq:variational-difference-quotient}, conditional
Cauchy--Schwarz, \eqref{eq:lower-order-FTC}, and
\eqref{eq:fixed-measure-candidate-moments} imply, for a finite
integer $M$ and every $\rho$,
\begin{align}
|\Delta_\varepsilon U_\rho^\alpha|^p
\le C_{k,j,p}\Bigg[1
&+\sum_{\ell=1}^j\bigl(|D_x^\ell X_\rho^x|^M
+|D_x^\ell X_\rho^y|^M\bigr)\notag\\
&+\int_0^1\sum_{\ell=1}^j
|D_x^\ell X_\rho^{x+\lambda\varepsilon e_i}|^M\dd\lambda
\Bigg].
\label{eq:difference-quotient-polynomial-envelope}
\end{align}
To see that no parameter-dependent constant enters this bound,
the conditional flow-difference estimate is used in each
instance of \eqref{eq:fixed-measure-candidate-moments}; along the
parameter segment,
\[
|X_\rho^z-X_\rho^x|
\le|\varepsilon|\int_0^1
|D_xX_\rho^{x+\lambda\varepsilon e_i}|\dd\lambda.
\]
All data quotients are integrals of derivatives of order at
most $j$. Jensen's inequality and a sufficiently large $M$
absorb their finite products into the right-hand side of
\eqref{eq:difference-quotient-polynomial-envelope}.
Use a countable dense set of $\rho$, path continuity, equivalence
of measures, and Lemma~\ref{lem:flow-estimates}. We obtain
\begin{equation}
\sup_{0<|\varepsilon|\le1}
\E^\Q\sup_{t\le s\le T}|\Delta_\varepsilon U_s^\alpha|^p
\le C_{k,j,p}.
\label{eq:difference-quotient-F-moment}
\end{equation}
Similarly, \eqref{eq:difference-quotient-source},
\eqref{eq:lower-order-FTC}, and
\eqref{eq:candidate-G-moment} give
\[
\sup_{0<|\varepsilon|\le1}
\E^\Q\left(\int_t^T|S_{\varepsilon,s}|\dd s\right)^p
\le C_{k,j,p}\br{x}^p.
\]
Under $\Q$, the drift in
\eqref{eq:variational-difference-quotient} additionally contains
$(H^x)^\top\Delta_\varepsilon V^\alpha$. Apply the localized
energy estimate, using \eqref{eq:difference-quotient-F-moment}
at exponent $2p$ and \eqref{eq:H-square-integrable}. It follows
that
\begin{equation}
\sup_{0<|\varepsilon|\le1}
\E^\Q\left(\int_t^T
(|\Delta_\varepsilon V_s^\alpha|^2
+|\Delta_\varepsilon\widetilde V_s^\alpha|^2)\dd s\right)^{p/2}
\le C_{k,j,p}\br{x}^p.
\label{eq:difference-quotient-G-moment}
\end{equation}
Equivalence transfers \eqref{eq:difference-quotient-Qhat-convergence}
to convergence in $\Q$-probability of the process norms.
For any prescribed $q\ge2$, choose $p>q$ in
\eqref{eq:difference-quotient-F-moment} and
\eqref{eq:difference-quotient-G-moment}. Vitali's theorem then
gives convergence in the $\Q$-based $\mathbb S^q\times\mathbb H^q
\times\mathbb H^q$ norms. Hence the candidates of order $j$
are the strong parameter derivatives. Their continuity was
already proved above, and induction reaches order $n_*$.
Substitution into \eqref{eq:candidate-F-moment} and
\eqref{eq:candidate-G-moment} proves
\eqref{eq:uniform-F-parameter}--\eqref{eq:uniform-G-parameter}.

Finally, continuity in the separable process spaces permits
jointly measurable versions, obtained by spatial grid
approximations. For the integrands these approximations are
predictable in $(\omega,s)$, so versions may be chosen
measurable with respect to the predictable sigma-field times
$\mathcal B(\R^d)$. Strong parameter differentiation and Fubini's
theorem identify the derivative fields as weak spatial
derivatives. Integrating
\eqref{eq:uniform-F-parameter}--\eqref{eq:uniform-G-parameter}
over $B_M$ proves \eqref{eq:local-parameter-Sobolev}. Equivalently,
the integrands have predictable versions with values in
$H^{n_*}(B_M)$; Sobolev embedding selects their spatially
continuous representatives outside a $\dd s\otimes\Q$-null set.
These representatives agree with the original integrand processes
at every deterministic parameter, in $\mathbb H^q$: spatial
averages converge to the representative by Sobolev embedding,
and to the original process by parameter continuity in
$\mathbb H^q$. This also makes evaluation at random spatial
parameters unambiguous.

The first-order estimate and the fundamental theorem of
calculus give, for every finite $q\ge2$,
\[
\E^\Q\sup_{t\le s\le T}|F_s^{t,x}-F_s^{t,z}|^q
\le C_{k,q}|x-z|^q.
\]
Choose $q>d$ and apply the Banach-valued Kolmogorov theorem
with target $C([t,T])$. This gives a version continuous in
$x$ with values in $C([t,T])$, hence jointly continuous in
$(s,x)$. On a countable dense set of spatial points,
\eqref{eq:F-representation} bounds the entire time path;
continuity extends the bound simultaneously to all $(s,x)$.
\end{proof}

\section{The existence and uniqueness of the BSPDE}
\label{sec:bspde}

In this section, we reconstruct $(f,g)$ from the untruncated
parameterized BSDE and the inverse stochastic flow.
We prove pointwise observation adaptation of $f$, observation
predictability of $g$, and the BSPDE identity directly for these
fields. We then establish weighted spatial estimates, joint
continuity and boundedness, and uniqueness in the solution class
of Definition~\ref{def:bspde-solution}.

\subsection{Flow multipliers and Eulerian reconstruction}

We first establish the properties of the stopped observation filtration and the weighted flow estimate needed for the reconstruction.

\begin{lemma}[Stopped reference observation filtration]
	\label{lem:reference-stopping-functional}
	Let 
	$
	\mathcal G_t^k:=\mathcal F^Y_{t\wedge\tau_k},
	\mathbb G^k=(\mathcal G_t^k)_{0\le t\le T}.
	$
	The filtration $\mathbb G^k$ has predictable representation with respect to
	$Y_{\cdot\wedge\tau_k}$. If $\theta\le\tau_k$ is an
	$\mathbb F^Y$-stopping time, it is also a $\mathbb G^k$-stopping time, and
	\begin{equation}
		\mathcal G_{t\wedge\theta}^k=\mathcal F^Y_{t\wedge\theta},
		\qquad \mathcal G_\theta^k=\mathcal F^Y_\theta.
		\label{eq:filtrations-agree-through-theta}
	\end{equation}
	Moreover, every $\mathbb G^k$-local martingale, viewed in $\mathbb F$, is a local
	martingale strongly orthogonal to $W$.
\end{lemma}
\begin{proof}

The predictable representation property is preserved under stopping; see
\citet[Chapter~III, Section~4]{JacodShiryaev2003}.  Indeed, let \(N\) be any square-integrable \(\mathbb G^k\)-martingale and set \(\xi:=N_T\). Since
$
\mathcal G_T^k=\mathcal F^Y_{\tau_k},
$
we have \(\xi\in L^2(\mathcal F^Y_{\tau_k})\). By the Brownian martingale representation theorem,
\[
M_t:=\mathbb E^\mathbb Q[\xi\mid\mathcal F_t^Y]
=M_0+\int_0^t Z_s^\top\,\mathrm dY_s
\]
for some square-integrable \(\mathbb F^Y\)-predictable process \(Z\). Optional sampling then gives
\[
\begin{aligned}
	N_t
	=\mathbb E^\mathbb Q[\xi\mid\mathcal G_t^k]
	=M_{t\wedge\tau_k}
	=N_0+\int_0^{t\wedge\tau_k}Z_s^\top\,\mathrm dY_s.
\end{aligned}
\]
This proves the predictable representation property of \(\mathbb G^k\) with respect to \(Y_{\cdot\wedge\tau_k}\).
Note that 
\eqref{eq:filtrations-agree-through-theta} is clearly true, and
\[
 \left[\int_0^{\cdot\wedge\tau_k}Z_s^\top\dd Y_s,W^\nu\right]_t=0,
 \qquad 1\le\nu\le r,
\]
so we finish the proof. 
\end{proof}

\begin{lemma}[Weighted stochastic-flow multiplier]
\label{lem:weighted-multiplier}
For $y\in\R^d$ and $0\le s\le\theta$, let $P_s(y)$ be any fixed
polynomial, with deterministic coefficients, in the components of
\[
\bigl\{D^j\Phi_s(y),\,
D^j\Psi_s(\Phi_s(y)):1\le j\le n_*\bigr\}
\quad
J_s(y),\;\text{and} \; \; J_s(y)^{-1}.
\]
Then, for every $m>0$ and every finite $q\ge1$,
\begin{equation}
 \left\|
   \sup_{0\le s\le\theta}
   \br{\Phi_s(y)}^{-m}J_s(y)|P_s(y)|
 \right\|_{L^q(\Q)}
 \le C_{m,q,k,P}\br{y}^{-m}.
 \label{eq:weighted-multiplier}
\end{equation}
\end{lemma}

\begin{proof}

Let  $V_m(x)=\br{x}^{-m}$, direct differentiation gives
\begin{align*}
 |DV_m(x)|\le C_mV_m(x)\br{x}^{-1}, \qquad
 |D^2V_m(x)|\le C_mV_m(x)\br{x}^{-2},
\end{align*}
and 
\begin{align*}
	\bigl|b(x,v)^\top DV_m(x)\bigr|
	+\left|\sum_{i,j=1}^d
	a^{ij}(x,v)D_{ij}V_m(x)\right|
	&\le C_m(1+|v|^2)V_m(x),\\
	\bigl|DV_m(x)\sigma(x,v)\bigr|^2
	&\le C_m(1+|v|^2)V_m(x)^2.
\end{align*}

Set $V_s=V_m(\Phi_s(y))$.  By It\^o's formula,
\[
\dd V_s=V_s\mu_s\dd s+V_s\nu_s^\top\dd W_s,
\]
where
\begin{align*}
	\mu_s
	&:=
	\frac{\chi_s^0}{V_m(\Phi_s(y))}
	(L_sV_m)(\Phi_s(y)),\\
	\nu_s^\ell
	&:=
	\frac{\chi_s^0}{V_m(\Phi_s(y))}
	\sum_{i=1}^d
	D_iV_m(\Phi_s(y))
	\sigma^{i\ell}(\Phi_s(y),u_s),
	\qquad \ell=1,\ldots,r.
\end{align*}
It follows that
\[
\int_0^T\bigl(|\mu_s|+|\nu_s|^2\bigr)\dd s
\le C_m\int_0^\theta(1+|u_s|^2)\dd s
\le C_m(T+k).
\]

Therefore, for every finite $q\ge1$, we have 
\begin{equation}
 \E^\Q\sup_{s\le\theta}V_m(\Phi_s(y))^q
 \le C_{m,q,k}V_m(y)^q.
 \label{eq:negative-flow-moment}
\end{equation}

All finite moments of
$\sup_{s\le\theta}|J_s(y)P_s(y)|$ are bounded uniformly in $y$ by
\eqref{eq:flow-derivative-moment} and \eqref{eq:maximal-inverse-flow-moments}.  Hence, H\"older's inequality and
\eqref{eq:negative-flow-moment} give
\[
 \left\|\sup_{s\le\theta}
 \br{\Phi_s(y)}^{-m}J_s(y)|P_s(y)|\right\|_{L^q}
 \le C_{m,q,k,P}\br{y}^{-m}.
\]

\end{proof}

Define the Eulerian fields directly from the untruncated BSDE by
\begin{equation}
 f_s(x):=F_s^{0,\Psi_s(x)},
 \qquad
 g_s(x):=G_s^{0,\Psi_s(x)}.
 \label{eq:reconstruct-limit}
\end{equation}
Here $(F^{0,y},G^{0,y},\widetilde G^{0,y})$ is the bounded solution of
\eqref{eq:untruncated-BSDE} with starting point $(0,y)$.
We use the jointly measurable parameter versions furnished by
Proposition~\ref{prop:uniform-parameter}. The spatial Sobolev regularity
in that proposition makes the compositions well defined. In particular,
for $\Q\otimes\dd s$-almost every $(\omega,s)$ the integrands have
continuous spatial Sobolev representatives on compact parameter sets;
we choose these representatives jointly predictably in $(\omega,s,y)$.
As explained in the proof of Proposition~\ref{prop:uniform-parameter},
at every deterministic parameter these representatives agree with
the original BSDE processes in the corresponding process norm.
The fields are
extended after $\theta$ by $f_s=\gamma$ and $g_s=0$.

The stochastic-flow property and uniqueness in the bounded class give,
for every deterministic $(t,x)$,
\begin{equation}
 F_s^{t,x}=f_s(X_s^{t,x}),\qquad
 G_s^{t,x}=g_s(X_s^{t,x}),\qquad t\le s\le T,
 \label{eq:decoupling-direct}
\end{equation}
where the first equality holds simultaneously for all $s$, almost surely,
and the second holds for $\Q\otimes\dd s$-almost every $(\omega,s)$.
To justify the random-parameter substitution, fix $(t,x)$ and let
$\xi=\Psi_t(x)$. Then $\xi$ is $\mathcal F_t$-measurable and
\[
 X_s^{t,x}=\Phi_s(\xi),\qquad
 \Psi_s(X_s^{t,x})=\xi,\qquad t\le s\le T.
\]
On $\{|\xi|\le M\}$, approximate $\xi$ by finite-valued
$\mathcal F_t$-measurable variables taking values in $B_{M+1}$.
For each approximating variable, the BSDE on $[t,T]$ follows by
finite pasting of deterministic-parameter equations. Sobolev embedding
in the parameter variable, the local parameter estimates of
Proposition~\ref{prop:uniform-parameter}, and the corresponding flow
estimates allow passage to the limit in the terminal value, drift, and
stochastic integrals. More explicitly, spatial Sobolev embedding gives
\[
 \E^\Q\int_t^T\sup_{y\in B_{M+1}}
 \bigl(|G_s^{0,y}|^2+|\widetilde G_s^{0,y}|^2
       +|H_s^{0,y}|^2\bigr)\dd s<\infty,
\]
using a slightly larger ball in the Sobolev norms. Spatial continuity
and dominated convergence therefore give convergence of the substituted
integrands in $L^2(\Q\otimes\dd s)$ and of the substituted drift in
$L^1(\Q\otimes\dd s)$, by Cauchy--Schwarz. The terminal values and first
components converge by joint continuity and the common bound on $F$;
BDG handles the stochastic integrals. Letting $M\uparrow\infty$ shows that
$(F^{0,\xi},G^{0,\xi},\widetilde G^{0,\xi})$, restricted to $[t,T]$,
solves the BSDE with coefficient $H^{t,x}$ and terminal value
$\Gamma^{t,x}$. This argument initially yields locally square-integrable
integrands. Since $|F^{0,\xi}|\le\|\gamma\|_\infty$ and
$\E^\Q\int_t^T|H_s^{t,x}|^2\dd s<\infty$, the localized energy estimate
in Lemma~\ref{lem:localized-BSDE-energy} makes those integrands globally
square integrable. Uniqueness in Proposition~\ref{prop:density-stability}
therefore identifies this triple with
$(F^{t,x},G^{t,x},\widetilde G^{t,x})$, proving
\eqref{eq:decoupling-direct}.

The proof of the following lemma is straightforward and is therefore omitted.
\begin{lemma}[Reference future filtration]
	\label{lem:future-filtration}
	For  $s\in[0,T]$, define
	$\mathbb K^s=(\mathcal K_r^s)_{r\in[s,T]}$ by
	\begin{equation}
		\mathcal K_r^s
		:=\mathcal F_s^Y
		\vee\sigma(W_a-W_s,Y_a-Y_s:s\le a\le r)
		\vee\mathcal N.
		\label{eq:future-filtration}
	\end{equation}
	Then the following statements hold.
	\begin{enumerate}
	
		\item[(i)]
		For every $r\in[s,T]$,
		$
		\mathcal F_r^Y\subset\mathcal K_r^s.
		$ 
		
		\item[(ii)]
		The random time $\rho_s:=s\vee\theta$ is a
		$\mathbb K^s$-stopping time, and
		$
		\mathcal F_\theta^Y\subset\mathcal K_{\rho_s}^s.
		$
	
	\end{enumerate}
\end{lemma}

By Lemma~\ref{lem:future-filtration} (i), the given
$\mathbb F^Y$-predictable control $u$, restricted to $[s,T]$,
is also $\mathbb K^s$-predictable.

\begin{proposition}[Diagonal stopped-observation measurability]
\label{prop:diagonal-Y}
The field $f$ in \eqref{eq:reconstruct-limit} has a jointly measurable
version whose sample paths are continuous in $(s,x)$. For this version,
for every deterministic $s\in[0,T]$ and $x\in\R^d$,
\begin{equation}
 f_{s\wedge\theta}(x)
 =\one_{\{s\le\theta\}}F_s^{s,x}
  +\one_{\{s>\theta\}}\gamma(x),
 \qquad \Q\text{-a.s.}
 \label{eq:diagonal-Y}
\end{equation}
Moreover, $f_{\cdot\wedge\theta}$ is a $\mathbb G^k$-adapted
$C(\R^d)$-valued process, with $C(\R^d)$ equipped with the topology
of uniform convergence on compact sets. In particular,
$f_{\cdot\wedge\theta}(x)$ is $\mathbb G^k$-adapted for every fixed
$x\in\R^d$.
\end{proposition}
\begin{proof}
For $q>d$, the first-derivative estimate in
Proposition~\ref{prop:uniform-parameter} gives
\[
 \E^\Q\sup_{0\le r\le T}|F_r^{0,y}-F_r^{0,z}|^q
 \le C_{k,q}|y-z|^q.
\]
The Banach-valued Kolmogorov criterion yields a version for which
$y\mapsto(F_r^{0,y})_{0\le r\le T}$ is continuous with values in
$C([0,T])$. The joint continuity of the inverse flow then makes
$f_s(x)=F_s^{0,\Psi_s(x)}$ continuous in $(s,x)$.

Fix deterministic $s,x$. In the future filtration $\mathbb K^s$ of
Lemma~\ref{lem:future-filtration}, construct $X^{s,x}$ and the likelihood
$Z^{s,x}$ exactly as in Proposition~\ref{prop:param-likelihood}.
The same entropy argument proves that this likelihood is a uniformly
integrable martingale in $\mathbb K^s$. The terminal value
$\Gamma^{s,x}=\gamma(X_\theta^{s,x})$ is
$\mathcal K_T^s$-measurable. Brownian martingale representation in
$\mathbb K^s$, with its initial sigma-algebra $\mathcal F_s^Y$, and the
ratio construction in Proposition~\ref{prop:density-stability} give a
bounded solution to \eqref{eq:untruncated-BSDE} in this filtration.
The localized energy estimate gives square-integrable integrands under
$\Q$. Full-filtration uniqueness identifies it with
$(F^{s,x},G^{s,x},\widetilde G^{s,x})$. Consequently,
$F_s^{s,x}$ is $\mathcal F_s^Y$-measurable.

Using \eqref{eq:decoupling-direct} at time $s$ and the frozen terminal
flow after $\theta$ gives
\[
 \one_{\{s\le\theta\}}f_s(x)
 =\one_{\{s\le\theta\}}F_s^{s,x},\qquad
 \one_{\{s>\theta\}}f_\theta(x)
 =\one_{\{s>\theta\}}\gamma(x),
\]
which proves \eqref{eq:diagonal-Y}. Applying this identity on a countable
dense set of spatial points and using spatial continuity shows that
$f_{s\wedge\theta}$ is an $\mathcal F_s^Y$-measurable
$C(\R^d)$-valued random variable. Thus $f_{\cdot\wedge\theta}$ is
a continuous $\mathbb F^Y$-adapted process. Since it is constant after
$\theta$, its value at time $s$ is $\mathcal F_{s\wedge\theta}^Y$-measurable.
By \eqref{eq:filtrations-agree-through-theta},
\[
 \mathcal F_{s\wedge\theta}^Y
 =\mathcal G_{s\wedge\theta}^k\subset\mathcal G_s^k.
\]
The claimed $\mathbb G^k$-adaptedness follows. Evaluation at any fixed
$x$ is a continuous map on $C(\R^d)$, and hence gives the final assertion.
\end{proof}

\subsection{Verification of the reconstructed BSPDE solution}

\begin{definition}[Localized It\^o--Wentzell condition]
	\label{def:localized-IW}
	Let $U$ be a scalar random field, let $A$ be scalar-valued, let $\tilde A$
	be $\R^m$-valued, and write $C=(C^1,\ldots,C^r)$. We say that
	$(U,A, \tilde A,C)$ satisfies condition $\mathrm{(IW)}_{n_*}$ on
	$[0,\theta]$ if, 
	\begin{equation}
		\mathrm{d}U_s
		=A_s\,\mathrm{d}s
		+\tilde A_s^\top\,\mathrm{d}Y_s
		+\sum_{\nu=1}^r C_s^\nu\,\mathrm{d}W_s^\nu.
		\label{eq:generic-parameter-semimartingale}
	\end{equation}
Here the identity holds as a semimartingale equality on $[0,\theta]$
after pairing with any $\varphi\in C_c^\infty(\R^d)$,
and $\tilde A, C$ admit predictable representatives.
	
	Moreover, for every $M<\infty$,   there exist stopping times
	$\kappa_j^M\uparrow\theta$, $\Q$-a.s., such that, for every $j$,
	\begin{align}
		&\E^\Q\operatorname*{ess\,sup}_{s\le\kappa_j^M}
		\|U_s\|_{H^{n_*}(B_M)}^2
		+\E^\Q\int_0^{\kappa_j^M}
		\|A_s\|_{H^{n_*-2}(B_M)}\,\mathrm{d}s
		\notag\\
		&\qquad
		+\E^\Q\int_0^{\kappa_j^M}
		\left(
		\|\tilde A_s\|_{H^{n_*-1}(B_M)}^2
		+\|C_s\|_{H^{n_*-1}(B_M)}^2
		\right)\mathrm{d}s
		<\infty.
		\label{eq:IW-local-Sobolev-norms}
	\end{align}
\end{definition}

In the remainder of this subsection, write
\[
b_s(x):=b(x,u_s),\qquad
\sigma_s(x):=\sigma(x,u_s),\qquad
a_s(x):=\sigma_s(x)\sigma_s(x)^\top,
\]
and let $\sigma_s^\nu$ denote the $\nu$th column of $\sigma_s$.  
We use the Einstein summation convention 
for spatial indices: every spatial index repeated within a term is 
summed from $1$ to $d$.

\begin{lemma}[Localized forward and inverse stochastic-flow formulas]
	\label{lem:flow-composition}
	
	\textnormal{(i) Inverse-flow dynamics.}
	For every component $\ell$, the inverse flow satisfies
	\begin{equation}
		\mathrm{d}\Psi_s^\ell(x)
		=\beta_s^\ell(x)\,\mathrm{d}s
		+\sum_{\nu=1}^r
		\zeta_s^{\ell\nu}(x)\,\mathrm{d}W_s^\nu,
		\label{eq:inverse-flow-SDE}
	\end{equation}
	where
	\begin{align}
		\zeta_s^{\ell\nu}(x)
		&=-D_i\Psi_s^\ell(x)\sigma_s^{i\nu}(x),
		\label{eq:inverse-flow-diffusion}\\
		\beta_s^\ell(x)
		&=-D_i\Psi_s^\ell(x)b_s^i(x)
		+\frac12a_s^{ip}(x)D_{ip}\Psi_s^\ell(x)
		\notag\\
		&\quad
		+\sum_{\nu=1}^r
		D_i\Psi_s^\ell(x)
		D_p\sigma_s^{i\nu}(x)\sigma_s^{p\nu}(x).
		\label{eq:inverse-flow-drift}
	\end{align}
	
	\textnormal{(ii) Inverse-flow composition.}
		Suppose that $(U,A,B,C)$ satisfies condition
	$\mathrm{(IW)}_{n_*}$.
	Then,
	\begin{equation}
		\mathrm{d}U_s(\Psi_s(x))
		=\widehat A_s\,\mathrm{d}s
		+\widehat B_s^\top\,\mathrm{d}Y_s
		+\sum_{\nu=1}^r
		\widehat C_s^\nu\,\mathrm{d}W_s^\nu,
		\label{eq:inverse-pullback-formula}
	\end{equation}
	where
	\begin{align}
		\widehat A_s(x)
		={}&A_s(\Psi_s(x))
		+D_\ell U_s(\Psi_s(x))\beta_s^\ell(x)
		\notag\\
		&+\frac12D_{\ell q}U_s(\Psi_s(x))
		\sum_{\nu=1}^r
		\zeta_s^{\ell\nu}(x)\zeta_s^{q\nu}(x)
		\notag\\
		&+\sum_{\nu=1}^r
		D_\ell C_s^\nu(\Psi_s(x))
		\zeta_s^{\ell\nu}(x),
		\label{eq:inverse-pullback-drift}\\
		\widehat B_s(x)
		={}&B_s(\Psi_s(x)), \quad
		\widehat C_s^\nu(x)
		=C_s^\nu(\Psi_s(x))
		-D_i\bigl[U_s(\Psi_s(x))\bigr]\sigma_s^{i\nu}(x).
		\label{eq:inverse-pullback-W}
	\end{align}

	\textnormal{(iii) Forward-flow composition.}
	Suppose that $(v,A,B,C)$ satisfies condition
	$\mathrm{(IW)}_{n_*}$. Then, 
	\begin{align}
		\mathrm{d}v_s(X_s^{t,x})
		={}&
		\left[
		A_s+L_sv_s
		+\sum_{\nu=1}^r
		D_iC_s^\nu\,\sigma_s^{i\nu}
		\right](X_s^{t,x})\,\mathrm{d}s
		+B_s(X_s^{t,x})^\top\,\mathrm{d}Y_s
		\notag\\
		&+\sum_{\nu=1}^r
		\left[
		C_s^\nu+D_iv_s\,\sigma_s^{i\nu}
		\right](X_s^{t,x})\,\mathrm{d}W_s^\nu.
		\label{eq:forward-flow-IW}
	\end{align}

	The inverse composition formula \eqref{eq:inverse-pullback-formula}
	is understood in the spatial weak sense.
	For each fixed deterministic $t\in[0,T]$, the forward composition
	formula \eqref{eq:forward-flow-IW} holds $\Q$-almost surely for
	Lebesgue-a.e.\ starting point $x\in\R^d$, simultaneously for all
	$s\in[t,\theta]$.
\end{lemma}

\begin{proof}

The coefficients of the inverse-flow dynamics are obtained by
applying the It\^o--Wentzell formula to the identity
$\Psi_s(\Phi_s(y))=y$; see \citet[Section~4.6]{Kunita1990}.
This gives 
\begin{align*}
 \zeta_s^{\ell\nu}(\Phi_s)
 +D_i\Psi_s^\ell(\Phi_s)\sigma_s^{i\nu}(\Phi_s)&=0,\\
 \beta_s^\ell(\Phi_s)+D_i\Psi_s^\ell(\Phi_s)b_s^i(\Phi_s)
 +\tfrac12a_s^{ip}(\Phi_s)D_{ip}\Psi_s^\ell(\Phi_s)
 +\sum_{\nu=1}^rD_i\zeta_s^{\ell\nu}(\Phi_s)
                  \sigma_s^{i\nu}(\Phi_s)&=0.
 \end{align*}
Substitution of
\[
 D_i\zeta_s^{\ell\nu}
 =-D_{ip}\Psi_s^\ell\sigma_s^{p\nu}
  -D_p\Psi_s^\ell D_i\sigma_s^{p\nu}
\]
proves \eqref{eq:inverse-flow-diffusion}-\eqref{eq:inverse-flow-drift}.

We now justify the passage from the classical smooth
It\^o--Wentzell formula to the localized Sobolev setting.  Fix $R<M$ and $j\ge1$.  Let
$\vartheta=\vartheta_{j,R,M}$ be the minimum of $\theta$,
$\kappa_j^{M+2}$ from \eqref{eq:IW-local-Sobolev-norms}, the time when
$\Psi_s(B_R)$ leaves $B_M$, and the time when
$\|\Psi_s\|_{C^2(B_R)}+\|\det D\Psi_s\|_{C(B_R)}
 +\|(\det D\Psi_s)^{-1}\|_{C(B_R)}$ exceeds $M$.  
 
  Choose
$
 0<\delta<\min\left\{1,n_*-2-\frac d2\right\}$ and define 
$ V^{\varepsilon,M}:=\varrho_\varepsilon*(\chi_MV),
 V\in\{U,A,B,C\},
$
where $\chi_M\in C_c^\infty(B_{M+2})$ satisfies
$\chi_M=1$ on $B_{M+1}$, and
$\varrho_\varepsilon(x):=\varepsilon^{-d}\varrho(x/\varepsilon)$
for a fixed nonnegative function $\varrho\in C_c^\infty(B_1)$
with $\int_{\R^d}\varrho(x)\,\mathrm{d}x=1$.
 
For $0<\varepsilon<1$, the standard
mollifier estimate and Sobolev embeddings
\citep[Chapter~4]{AdamsFournier2003} give
\begin{align*}
 \|U_s^{\varepsilon,M}-U_s\|_{C^2(B_M)}
 &\le C_M\varepsilon^\delta\|U_s\|_{H^{n_*}(B_{M+2})},\\
 \|A_s^{\varepsilon,M}-A_s\|_{C(B_M)}
 &\le C_M\varepsilon^\delta\|A_s\|_{H^{n_*-2}(B_{M+2})},\\
 \|B_s^{\varepsilon,M}-B_s\|_{C(B_M)}
 &\le C_M\varepsilon^\delta
   \|B_s\|_{H^{n_*-1}(B_{M+2})},\\
 \|C_s^{\varepsilon,M}-C_s\|_{C^1(B_M)}
 &\le C_M\varepsilon^\delta
   \|C_s\|_{H^{n_*-1}(B_{M+2})}.
\end{align*}
Hence, by \eqref{eq:IW-local-Sobolev-norms}, as $\varepsilon \to 0$, we have  
\begin{align}
 &\E^\Q\esssup_{s\le\vartheta}
   \|U_s^{\varepsilon,M}-U_s\|_{C^2(B_M)}^2
 +\E^\Q\int_0^\vartheta
   \|A_s^{\varepsilon,M}-A_s\|_{C(B_M)}\dd s
 \notag\\
 &\quad+\E^\Q\int_0^\vartheta
 \left(\|B_s^{\varepsilon,M}-B_s\|_{C(B_M)}^2
 +\|C_s^{\varepsilon,M}-C_s\|_{C^1(B_M)}^2\right)\dd s
 \longrightarrow0.
 \label{eq:IW-mollification-limit}
\end{align}

Applying the classical formula
\citep[Theorem~3.1]{Krylov2011} to
$U_s^{\varepsilon,M}(\Psi_s(x))$ yields
\eqref{eq:inverse-pullback-formula} with $U$ replaced by
$U^{\varepsilon,M}$ and coefficients
$\widehat A^{\varepsilon,M}$, $\widehat B^{\varepsilon,M}$,
and $\widehat C^{\varepsilon,M}$, obtained by substituting
the mollified fields into
\eqref{eq:inverse-pullback-drift}--\eqref{eq:inverse-pullback-W}.

On the stochastic interval $[0,\vartheta]$, the composition map
$f\mapsto f\circ\Psi_s$ is bounded from $C^q(B_M)$ to
$C^q(B_R)$, uniformly in $s$, for $q=0,1,2$, and
\[
 \sup_{x\in B_R}(|\beta_s(x)|+|\zeta_s(x)|^2)
 \le C_{R,M}(1+|u_s|^2),\qquad
 \int_0^\theta(1+|u_s|^2)\dd s\le T+k.
\]
Thus, \eqref{eq:IW-mollification-limit}, Cauchy--Schwarz, and BDG give:
for $\varphi\in C_c^\infty(B_R)$,
\begin{align*}
 \E^\Q\int_0^\vartheta
  |\langle\widehat A_s^{\varepsilon,M}-\widehat A_s,
          \varphi\rangle|\dd s&\longrightarrow0,\\
 \E^\Q\sup_{0\le q\le\vartheta}\left|
  \int_0^q\langle\widehat B_s^{\varepsilon,M}-\widehat B_s,
          \varphi\rangle^\top\dd Y_s
 +\sum_{\nu=1}^r\int_0^q
  \langle\widehat C_s^{\varepsilon,M,\nu}-\widehat C_s^\nu,
          \varphi\rangle\dd W_s^\nu\right|^2
 &\longrightarrow0.
\end{align*}
Letting $\varepsilon\downarrow0$ proves
\eqref{eq:inverse-pullback-formula}-\eqref{eq:inverse-pullback-W}.  The same localization and mollification argument with $\Phi_{t,s}$
proves \eqref{eq:forward-flow-IW}.

Fix $R<\infty$, for each fixed integer $M>R$, condition
$\mathrm{(IW)}_{n_*}$ gives
\[
\kappa_j^{M+2}\uparrow\theta
\qquad\text{as }j\to\infty,\quad\Q\text{-a.s.}
\]
Thus, letting $j\to\infty$ removes the localization imposed by
$\kappa_j^{M+2}$.

Next, the inverse-flow regularity in Lemma~\ref{98} implies that
\[
\sup_{s\le\theta}\|\Psi_s\|_{C^2(B_R)}
+\sup_{s\le\theta}\bigl(
\|\det D\Psi_s\|_{C(B_R)}
+\|(\det D\Psi_s)^{-1}\|_{C(B_R)}
\bigr)<\infty
\qquad\Q\text{-a.s.}
\]
Consequently, for almost every sample path, all sufficiently large
$M$ ensure that $\Psi_s(B_R)\subset B_M$ and that the above norms
remain below $M$ throughout $[0,\theta]$. Hence
\[
\lim_{M\to\infty}\lim_{j\to\infty}
\vartheta_{j,R,M}=\theta
\qquad\Q\text{-a.s.}
\]
This removes the localizations, continuity extends the identities
to the terminal time $\theta$ and completes the proof.

The above convergence estimates, together with stochastic Fubini
\citep[Lemma~2.6]{Krylov2011}, establish the inverse composition
formula \eqref{eq:inverse-pullback-formula} in the spatial weak sense.

For each fixed deterministic $t\in[0,T]$, Fubini's theorem and
continuity in time, followed by a countable exhaustion of $\R^d$
by balls, show that the integral form of
\eqref{eq:forward-flow-IW} holds $\Q$-almost surely for
Lebesgue-a.e.\ starting point $x\in\R^d$, simultaneously for all
$s\in[t,\theta]$.

\end{proof}

\begin{proposition}[Stochastic-characteristic equation]
\label{prop:characteristic-equation}
The fields in \eqref{eq:reconstruct-limit} satisfy
\begin{equation}
 \dd f_s(x)
 =-\left[L_sf_s(x)+h(x,u_s)^\top g_s(x)\right]\dd s
   +g_s(x)^\top\dd Y_s,\qquad f_\theta=\gamma,
 \label{eq:characteristic-BSPDE}
\end{equation}
in the spatial weak sense on $[0,\theta]$. Moreover,
\begin{equation}
 \widetilde G_s^{0,y}
 =\sigma(\Phi_s(y),u_s)^\top\nabla f_s(\Phi_s(y))
 \quad\text{for }\dd\Q\otimes\dd s\otimes\dd y\text{-a.e. }(\omega,s,y)
 \text{ on }[0,\theta],
 \label{eq:tilde-identification}
\end{equation}
and $g$, extended by zero after $\theta$, admits a jointly measurable
$\mathbb G^k$-predictable version, and hence an
$\mathbb F^Y$-predictable version, locally square integrable in space.
\end{proposition}
\begin{proof}
Set, on $[0,\theta]\times\R^d$,
\[
 U_s(y)=F_s^{0,y},\quad
 A_s(y)=-(H_s^{0,y})^\top G_s^{0,y},\quad
 B_s(y)=G_s^{0,y},\quad C_s(y)=\widetilde G_s^{0,y}.
\]
For every finite $M$, Proposition~\ref{prop:uniform-parameter} and
Fubini give
\[
 \E^\Q\esssup_{s\le\theta}\|U_s\|_{H^{n_*}(B_M)}^2
 +\E^\Q\int_0^\theta
   \bigl(\|B_s\|_{H^{n_*-1}(B_M)}^2
         +\|C_s\|_{H^{n_*-1}(B_M)}^2\bigr)\dd s<\infty.
\]
Write $a=n_*-2>d/2$. The growth assumption and the stopped control
energy give the order-zero part of
\[
 \E^\Q\int_0^\theta\|H_s^{0,\cdot}\|_{H^a(B_M)}^2\dd s<\infty;
\]
the positive-order parts follow from the chain rule, the bounded spatial
derivatives of $h$, and Lemma~\ref{lem:flow-estimates}.
Since $H^a(B_M)$ is a multiplication algebra, Cauchy--Schwarz yields
\begin{align*}
 \E^\Q\int_0^\theta\|A_s\|_{H^a(B_M)}\dd s
 &\le C_M\left(\E^\Q\int_0^\theta
                   \|H_s^{0,\cdot}\|_{H^a(B_M)}^2\dd s\right)^{1/2}\\
 &\qquad\times\left(\E^\Q\int_0^\theta
                   \|G_s^{0,\cdot}\|_{H^a(B_M)}^2\dd s\right)^{1/2}
 <\infty.
\end{align*}
Thus only time-$L^1$ integrability of the drift is used. Pairing
\eqref{eq:untruncated-BSDE} with a test function and applying stochastic
Fubini proves \eqref{eq:generic-parameter-semimartingale}. Hence
$(U,A,B,C)$ satisfies $\mathrm{(IW)}_{n_*}$.

The inverse formula of Lemma~\ref{lem:flow-composition} gives
\begin{equation}
 \dd f_s(x)
 =\mathfrak a_s(x)\dd s+g_s(x)^\top\dd Y_s
  +q_s(x)^\top\dd W_s,
 \label{eq:preliminary-eulerian-decomp}
\end{equation}
where $\mathfrak a$ is \eqref{eq:inverse-pullback-drift} for the above
fields, and
\begin{equation}
 q_s^\nu(x)=\widetilde G_s^{0,\Psi_s(x),\nu}
            -D_if_s(x)\sigma_s^{i\nu}(x),
 \qquad \nu=1,\ldots,r.
 \label{eq:q-explicit}
\end{equation}
Enlarge the spatial localizations in Lemma~\ref{lem:flow-composition}
to bound the derivatives of $\Phi$ and $\Psi$ through order $n_*$
and the reciprocal Jacobians on the relevant compact sets. These
stopping times still increase to $\theta$, by the flow regularity.
On a stopped interval where $\Psi_s(B_R)\subset B_M$ and these
bounds are at most $L$, the Sobolev chain rule gives, with a
constant depending on $L,R,M$,
\begin{align*}
 \|f_s\|_{H^{n_*}(B_R)}
 &\le C_{L,R,M}\|U_s\|_{H^{n_*}(B_M)},\\
 \|g_s\|_{H^{n_*-1}(B_R)}
 &\le C_{L,R,M}\|B_s\|_{H^{n_*-1}(B_M)},\\
 \|q_s\|_{H^{n_*-1}(B_R)}
 &\le C_{L,R,M}\bigl[
       \|C_s\|_{H^{n_*-1}(B_M)}
       +(1+|u_s|)\|U_s\|_{H^{n_*}(B_M)}\bigr],\\
 \|\mathfrak a_s\|_{H^{n_*-2}(B_R)}
 &\le C_{L,R,M}\bigl[
       \|A_s\|_{H^{n_*-2}(B_M)}
       +(1+|u_s|^2)\|U_s\|_{H^{n_*}(B_M)}\\
 &\hspace{42mm}+(1+|u_s|)\|C_s\|_{H^{n_*-1}(B_M)}\bigr].
\end{align*}
The parameter-field bounds, $\int_0^\theta(1+|u_s|^2)\dd s\le T+k$,
and Cauchy--Schwarz give the Eulerian condition
$\mathrm{(IW)}_{n_*}$ for $(f,\mathfrak a,g,q)$.
In particular, $g,q$ are locally square integrable in space and time.
In fact, for every
$R<\infty$, changing variables $x=\Phi_s(y)$, applying
Lemma~\ref{lem:weighted-multiplier}, and using
\eqref{eq:uniform-G-parameter} with $q=4$ give, for $a>d+2$,
\[
 \E^\Q\int_0^\theta\int_{B_R}|g_s(x)|^2\dd x\dd s
 \le C_{R,a}\int_{\R^d}\br{y}^{2-a}\dd y<\infty.
\]
Here $\one_{\{|\Phi_s(y)|\le R\}}\le C_{R,a}\br{\Phi_s(y)}^{-a}$,
and H\"older separates the squared time norm of $G^{0,y}$ from
$\sup_{s\le\theta}\br{\Phi_s(y)}^{-a}J_s(y)$.

Choose a countable subset $\mathcal D\subset C_c^\infty(\R^d)$ such that
$\mathcal D\cap C_c^\infty(B_R)$ contains an orthonormal basis of
$L^2(B_R)$ for every integer $R\ge1$; adjoining a countable dense set of
test functions if necessary, take its rational linear span.
For $\varphi\in\mathcal D$ let
$Z_t^\varphi=\langle f_{t\wedge\theta},\varphi\rangle$.
Proposition~\ref{prop:diagonal-Y} and Stricker's theorem
\citep[Theorem~2.1]{FollmerProtter2011} make $Z^\varphi$ a
$\mathbb G^k$-semimartingale. By
Lemma~\ref{lem:reference-stopping-functional}, its local martingale
part is an integral with respect to $Y_{\cdot\wedge\theta}$.
It is therefore orthogonal to $W$, so
\begin{equation}
 0=[Z^\varphi,W^\nu]_t
  =\int_0^{t\wedge\theta}\langle q_s^\nu,\varphi\rangle\dd s.
 \label{eq:q-bracket-zero}
\end{equation}
Countability, density, and local square integrability give
$q=0$ for $\dd\Q\otimes\dd s\otimes\dd x$-almost every point on
$[0,\theta]$.

Let $K^\varphi$ be the $\mathbb G^k$-predictable integrand of the local
martingale part of $Z^\varphi$. Computing the covariation with $Y^j$
from its two decompositions gives
\[
 \int_0^{t\wedge\theta}\langle g_s^j,\varphi\rangle\dd s
 =[Z^\varphi,Y^j]_t
 =\int_0^{t\wedge\theta}K_s^{\varphi,j}\dd s.
\]
Consequently each pairing $\langle g_s,\varphi\rangle$ has a
$\mathbb G^k$-predictable version. For an orthonormal basis
$(e_i^R)_{i\ge1}\subset\mathcal D\cap C_c^\infty(B_R)$, Parseval gives
\[
 \sum_{i=1}^J\langle g_s,e_i^R\rangle e_i^R
 \longrightarrow g_s
 \quad\text{in }L^2(\Omega\times[0,\theta]\times B_R).
\]
Closedness of the predictable subspace produces an
$L^2(B_R;\R^m)$-valued predictable representative. Representatives on
the integer balls agree almost everywhere on overlaps, and hence may
be chosen consistently to give a jointly measurable predictable field.

Since $q=0$, the forward formula of
Lemma~\ref{lem:flow-composition} now yields
\begin{align*}
 \dd F_s^{0,y}
 ={}&[\mathfrak a_s+L_sf_s](\Phi_s(y))\dd s
     +g_s(\Phi_s(y))^\top\dd Y_s\\
 &+[\sigma(\Phi_s(y),u_s)^\top\nabla f_s(\Phi_s(y))]^\top\dd W_s.
\end{align*}
Comparison with \eqref{eq:untruncated-BSDE}, uniqueness of the
semimartingale decomposition, and change of variables by $\Phi_s$ give
\[
 \mathfrak a_s=-L_sf_s-h(\cdot,u_s)^\top g_s,
 \qquad
 \widetilde G_s^{0,y}
 =\sigma(\Phi_s(y),u_s)^\top\nabla f_s(\Phi_s(y)).
\]
Finally, the terminal equation, first on a countable dense set of
parameters and then by spatial continuity, gives
\[
 f_\theta(\Phi_\theta(y))=F_\theta^{0,y}
 =\gamma(\Phi_\theta(y)),\qquad y\in\R^d.
\]
The diffeomorphism property proves $f_\theta=\gamma$.
This establishes \eqref{eq:characteristic-BSPDE} and
\eqref{eq:tilde-identification}. Integrating the former against a test
function gives \eqref{eq:weak-bspde}; continuity in time and a
countable dense set of test functions give a common event on which
the weak identity holds for all times and test functions.
\end{proof}

\subsection{Weighted Sobolev estimates}

For $|\alpha|\le n_*$ and $(v,V)=(f,F)$ or $(g,G)$,
the spatial chain rule gives
\begin{equation}
 (D_x^\alpha v_s)(\Phi_s(y))
 =\sum_{|\beta|\le|\alpha|}
   D_y^\beta V_s^{0,y}\,P_{\alpha,\beta,s}(y),
 \label{eq:chain-f}
\end{equation}
where $P_{\alpha,\beta,s}(y)$ is a polynomial in the components
of $D_y^j\Phi_s(y)$, $1\le j\le|\alpha|$, and the entries of
$[D_y\Phi_s(y)]^{-1}$. This identity is understood for weak
spatial derivatives and holds for $\dd\Q\otimes\dd s\otimes\dd y$
almost every $(\omega,s,y)$ with $s\le\theta$.

\begin{proposition}[Weighted Sobolev estimate]
\label{prop:weighted-Sobolev}
Let $(f,g)$ be the fields defined in
\eqref{eq:reconstruct-limit}. Set $\ell_0=d/2+2$. Then
\begin{equation}
 \E^\Q\esssup_{s\le\theta}\norm{f_s}_{H_{\ell_0}^{n_*}}^2
 +\E^\Q\int_0^\theta
    \norm{g_s}_{H_{\ell_0}^{n_*}(\R^d;\R^m)}^2\dd s
 <\infty.
 \label{eq:weighted-Sobolev-estimate}
\end{equation}
\end{proposition}

\begin{proof}
Fix $\ell>0$ and $|\alpha|\le n_*$. Using the change of variables
$x=\Phi_s(y)$, with $\dd x=J_s(y)\dd y$, and
\eqref{eq:chain-f}, we obtain
\begin{align*}
 &\int_{\R^d}|D_x^\alpha v_s(x)|^2\br{x}^{-2\ell}\dd x\\
 &\qquad\le C_\alpha\sum_{|\beta|\le|\alpha|}
 \int_{\R^d}|D_y^\beta V_s^{0,y}|^2
 |P_{\alpha,\beta,s}(y)|^2
 \br{\Phi_s(y)}^{-2\ell}J_s(y)\dd y.
\end{align*}
Write
\[
 A_{\alpha,\beta}(y)
 :=\sup_{s\le\theta}
 |P_{\alpha,\beta,s}(y)|^2
 \br{\Phi_s(y)}^{-2\ell}J_s(y).
\]
Lemma~\ref{lem:weighted-multiplier}, with $m=2\ell$ and $q=2$,
implies
\[
 \|A_{\alpha,\beta}(y)\|_{L^2(\Q)}
 \le C_{k,\ell,\alpha,\beta}\br{y}^{-2\ell}.
\]
For the first component, Tonelli's theorem, H\"older's inequality,
and the $q=4$ estimate in Proposition~\ref{prop:uniform-parameter}
give
\begin{align*}
 &\E^\Q\esssup_{s\le\theta}
 \int_{\R^d}|D_x^\alpha f_s(x)|^2\br{x}^{-2\ell}\dd x\\
 &\quad\le C_\alpha\sum_{|\beta|\le|\alpha|}
 \int_{\R^d}\|A_{\alpha,\beta}(y)\|_{L^2(\Q)}
 \left(\E^\Q\sup_{s\le\theta}|D_y^\beta F_s^{0,y}|^4\right)^{1/2}
 \dd y\\
 &\quad\le C_{k,\ell,\alpha}\int_{\R^d}\br{y}^{-2\ell}\dd y.
\end{align*}
For the second component, the same argument applied to the
time integral gives
\begin{align*}
 &\E^\Q\int_0^\theta\int_{\R^d}
 |D_x^\alpha g_s(x)|^2\br{x}^{-2\ell}\dd x\dd s\\
 &\quad\le C_\alpha\sum_{|\beta|\le|\alpha|}
 \int_{\R^d}\|A_{\alpha,\beta}(y)\|_{L^2(\Q)}
 \left\{\E^\Q
 \left(\int_0^\theta|D_y^\beta G_s^{0,y}|^2\dd s\right)^2
 \right\}^{1/2}\dd y\\
 &\quad\le C_{k,\ell,\alpha}
 \int_{\R^d}\br{y}^{2-2\ell}\dd y.
\end{align*}
The factor $\br{y}^2$ in the last line is precisely the
square root of the $q=4$ growth bound for the martingale
integrands in Proposition~\ref{prop:uniform-parameter}.
Summing over $|\alpha|\le n_*$ and taking $\ell=\ell_0$ yields
\begin{align}
 &\E^\Q\esssup_{s\le\theta}\norm{f_s}_{H_{\ell_0}^{n_*}}^2
 +\E^\Q\int_0^\theta\norm{g_s}_{H_{\ell_0}^{n_*}}^2\dd s
 \notag\\
 &\qquad\le C_k\int_{\R^d}
 \left(\br{y}^{-2\ell_0}+\br{y}^{2-2\ell_0}\right)\dd y<\infty,
 \label{eq:explicit-weight-integral}
\end{align}
since $2\ell_0=d+4$. This proves
\eqref{eq:weighted-Sobolev-estimate} directly for $(f,g)$.
\end{proof}

\begin{proposition}[Weighted pointwise estimates]
\label{prop:weighted-pointwise}
With $\ell_0=d/2+2$, for every $\lambda\ge\ell_0$ and every
$|\alpha|\le2$, \eqref{eq:weighted-sup-2} holds.
\end{proposition}
\begin{proof}
	Fix a multi-index $\alpha$ with $|\alpha|\le2$, and let
	$v\in H_{\ell_0}^{n_*}(\R^d)$.
	For $z\in\mathbb Z^d$, define $Q_z=z+(-1,1)^d$.
	Since $n_*-|\alpha|>d/2$, the Sobolev embedding
	\citep[Theorem~4.12, Part~I, Case~A, p.~85]{AdamsFournier2003}
	gives
	\begin{equation}
		\sup_{x\in Q_z}|D^\alpha v(x)|^2
		\le C\sum_{|\beta|\le n_*}
		\int_{Q_z}|D^\beta v(x)|^2\,\mathrm{d}x,
		\label{eq:local-Sobolev}
	\end{equation}
	where we use the continuous representative of $D^\alpha v$.
	By translation, the constant $C$ is independent of $z$.
	
	There exists a constant $c\ge1$, independent of $z$, such that
	\[
	c^{-1}\br{z}\le\br{x}\le c\br{z},
	\qquad x\in Q_z.
	\]
	Since the cubes $(Q_z)_{z\in\mathbb Z^d}$ cover $\R^d$,
	\eqref{eq:local-Sobolev} implies
	\begin{align*}
		\sup_{x\in\R^d}
		\frac{|D^\alpha v(x)|^2}{\br{x}^{2\ell_0}}
		&\le C\sup_{z\in\mathbb Z^d}\br{z}^{-2\ell_0}
		\sum_{|\beta|\le n_*}
		\int_{Q_z}|D^\beta v(x)|^2\,\mathrm{d}x\\
		&\le C\sup_{z\in\mathbb Z^d}
		\sum_{|\beta|\le n_*}
		\int_{Q_z}
		|D^\beta v(x)|^2\br{x}^{-2\ell_0}\,\mathrm{d}x\\
		&\le C\norm{v}_{H_{\ell_0}^{n_*}}^2.
	\end{align*}
	The same estimate holds for $\R^m$-valued functions by applying
	the scalar estimate to each component and summing.
	
	Applying this estimate to $v=f_s(\omega,\cdot)$ and
	$v=g_s(\omega,\cdot)$, taking the essential supremum in time
	for $f$ and integrating in time for $g$, and then taking
	expectations, Proposition~\ref{prop:weighted-Sobolev} yields
	\eqref{eq:weighted-sup-2} for $\lambda=\ell_0$.
	Finally, for $\lambda\ge\ell_0$,
	$\br{x}^{-2\lambda}\le\br{x}^{-2\ell_0}$, which proves the
	remaining cases.
\end{proof}

\begin{proposition}[Local time continuity of $f$]
	\label{prop:local-time-continuity}
	Let $(f,g)$ be a weighted Sobolev solution of
	\eqref{eq:intro-dual-bspde}.
	Then $f$ admits a version such that, $\Q$-almost surely,
	for every $R<\infty$,
	\begin{equation}
		f\in C\bigl([0,\theta];H^{n_*-2}(B_R)\bigr)
		\hookrightarrow C\bigl([0,\theta];C(\overline B_R)\bigr),
		\label{eq:local-time-continuity}
	\end{equation}
\end{proposition}

\begin{proof}
	Fix $R<\infty$. By the coefficient assumptions and Leibniz's rule,
	\begin{align*}
		\|L_sf_s\|_{H^{n_*-2}(B_R)}
		&\le C_R(1+|u_s|^2)\|f_s\|_{H^{n_*}(B_{R})},\\
		\|h(\cdot,u_s)^\top g_s\|_{H^{n_*-2}(B_R)}
		&\le C_R(1+|u_s|)\|g_s\|_{H^{n_*-2}(B_{R})}.
	\end{align*}
	Together with \eqref{eq:solution-norm}, Cauchy--Schwarz, and
	$\int_0^\theta |u_s|^2\dd s\le k$, these estimates imply
	\begin{equation}\label{xp}
	\int_0^\theta
	\left(
	\|L_sf_s\|_{H^{n_*-2}(B_R)}
	+\|h(\cdot,u_s)^\top g_s\|_{H^{n_*-2}(B_R)}
	\right)\dd s<\infty
	\qquad\Q\text{-a.s.}
	\end{equation}
	Moreover, \eqref{eq:solution-norm} gives
	\begin{equation}\label{xp1}
	\E^\Q\int_0^\theta
	\|g_s\|_{H^{n_*-2}(B_R;\R^m)}^2\dd s<\infty.
	\end{equation}
By \eqref{xp} and \eqref{xp1},	the drift integral has continuous paths in $H^{n_*-2}(B_R)$,
	so does the Hilbert-valued stochastic integral
	\citep[Chapter~4]{DaPratoZabczyk2014}.
	Hence, the weak identity defines
	an $H^{n_*-2}(B_R)$-continuous modification of $f$.

Choose the modifications consistently for $R\in\mathbb N$, agreement
almost everywhere and time continuity imply indistinguishability on
overlaps. Finally,
$n_*-2>d/2$ and the Sobolev embedding give
\eqref{eq:local-time-continuity}.

\end{proof}

\subsection{Boundedness and uniqueness}

\begin{proposition}[Boundedness of the first component]
\label{prop:f-bound}
Let $f$ be the continuous version in
Proposition~\ref{prop:diagonal-Y}. There is an event $\Omega_0$ with
$\Q(\Omega_0)=1$ such that
\begin{equation}
 |f_s(\omega,x)|\le\|\gamma\|_\infty
 \quad\text{for every }\omega\in\Omega_0,
 \ 0\le s\le\theta(\omega),\ x\in\R^d.
 \label{eq:f-bound-common-version}
\end{equation}
\end{proposition}

\begin{proof}
For each deterministic $y$, \eqref{eq:F-representation} gives
$|F_s^{0,y}|\le\|\gamma\|_\infty$ almost surely for every
deterministic $s$. Choose the jointly continuous version of
$F^{0,\cdot}$ used in Proposition~\ref{prop:diagonal-Y}.
Taking a countable intersection over rational $s\in[0,T]$
and $y\in\Q^d$, and then using continuity, gives a single event
of probability one on which
\[
 |F_s^{0,y}|\le\|\gamma\|_\infty
 \qquad\text{for every }s\in[0,T],\ y\in\R^d.
\]
Intersect this event with the probability-one event on which
the inverse flow and the reconstruction identity hold.
Then \eqref{eq:reconstruct-limit} yields, simultaneously for
all $s\le\theta$ and $x\in\R^d$,
\[
 |f_s(x)|=|F_s^{0,\Psi_s(x)}|\le\|\gamma\|_\infty.
\]
This proves \eqref{eq:f-bound-common-version}.
\end{proof}

\begin{proposition}[Uniqueness in the natural bounded observation-adapted class]
\label{prop:bspde-uniqueness}
Suppose Assumptions~\ref{ass:filtering} and
\ref{ass:bspde-regularity} hold.
Fix $k\in\N$, an $\mathbb F^Y$-stopping time $\theta\le\tau_k$,
and $\gamma\in\mathscr G_\theta^{n_*+1}$.
Then \eqref{eq:intro-dual-bspde} has at most one weighted
Sobolev solution $(f,g)$ whose first component $f$ belongs to
$L^\infty(\Omega\times[0,\theta]\times\R^d)$.
\end{proposition}

\begin{proof}
Let $\bar f=f^1-f^2$ and $\bar g=g^1-g^2$, taking the
locally space--time continuous versions of $f^1,f^2$ provided
by Proposition~\ref{prop:local-time-continuity}. Their essential
boundedness and continuity imply a common bound at all
$(s,x)\in[0,\theta]\times\R^d$ outside one $\Q$-null event.

Fix $t\in[0,T]$. On $A_t$, for $t\le s\le\theta$, define
\begin{equation}
 \begin{aligned}
 \bar F_s^x&:=\bar f_s(X_s^{t,x}),\qquad
 \bar G_s^x:=\bar g_s(X_s^{t,x}),\\
 \overline{\widetilde G}_s^x
 &:=\sigma(X_s^{t,x},u_s)^\top
 \nabla\bar f_s(X_s^{t,x}).
 \end{aligned}
 \label{eq:uniqueness-characteristic-fields}
\end{equation}
Extend $\one_{A_t}\bar F^x$ by zero after $\theta$, and the
two integrands by zero outside the active interval.
By the forward formula \eqref{eq:forward-flow-IW} in
Lemma~\ref{lem:flow-composition}, for Lebesgue-a.e. $x$,
\begin{equation}
 \dd\bar F_s^x
 =-(H_s^{t,x})^\top\bar G_s^x\dd s
 +(\bar G_s^x)^\top\dd Y_s
 +(\overline{\widetilde G}_s^x)^\top\dd W_s,
 \qquad\bar F_\theta^x=0
 \quad\text{on }A_t.
 \label{eq:zero-terminal-flow-BSDE}
\end{equation}
The hypotheses of the localized formula follow from the
solution norm exactly as in
Proposition~\ref{prop:local-time-continuity}, with
$A=-L_s\bar f_s-h(\cdot,u_s)^\top\bar g_s$,
$B=\bar g$, and $C=0$.

We only need local square integrability of the characteristic
integrands. Each continuous trajectory $X^{t,x}$ lies in
some random compact ball on $[t,\theta]$. Local Sobolev
embedding and the solution norms imply, almost surely for
every integer $R$,
\[
 \int_0^\theta\sup_{z\in B_R}|\bar g_s(z)|^2\dd s<\infty,
 \qquad
 \esssup_{s\le\theta}\sup_{z\in B_R}
 |\nabla\bar f_s(z)|^2<\infty.
\]
Together with the linear growth of $\sigma$ and
$\int_0^\theta|u_s|^2\dd s\le k$, this proves that the
two stochastic integrands in
\eqref{eq:zero-terminal-flow-BSDE} have finite quadratic
variation after localization. The same is true under the
equivalent measure $\widehat\Q^{t,x}$.

Under this measure,
\[
 \widehat Y_s^{t,x}
 :=Y_s-\int_0^sH_a^{t,x}\dd a
\]
and $W$ are Brownian motions, so the drift in
\eqref{eq:zero-terminal-flow-BSDE} disappears. Consequently,
$\one_{A_t}\bar F^x$, with its zero extension after $\theta$,
is a bounded continuous local martingale on $[t,T]$ with
terminal value zero. It is therefore a true martingale and
identically zero. Its quadratic variation also vanishes.
By equivalence of the two measures, we obtain
\begin{equation}
 \begin{aligned}
 &\one_{A_t}\bar F_s^x=0
 \quad(t\le s\le\theta),\qquad\Q\text{-a.s.},\\
 &\one_{A_t}\bar G^x=0,\qquad
 \one_{A_t}\overline{\widetilde G}^{x}=0,
 \quad\dd\Q\otimes\dd s\text{-a.e.},
 \end{aligned}
 \label{eq:characteristic-zero}
\end{equation}
for Lebesgue-a.e. $x$.

Since $X_t^{t,x}=x$ on $A_t$ and
$\bar f_{t\wedge\theta}=\bar f_\theta=0$ on $A_t^c$,
\begin{equation}
 \E^\Q\int_{B_R}|\bar f_{t\wedge\theta}(x)|^2\dd x=0.
 \label{eq:integrated-f-zero}
\end{equation}
Taking a countable intersection over rational $t\in[0,T]$
and integer $R\ge1$, and using local space--time continuity,
gives $\bar f=0$ everywhere on $[0,\theta]\times\R^d$
outside one $\Q$-null event.

Taking $t=0$ in \eqref{eq:characteristic-zero} gives
$\bar g_s(\Phi_s(x))=0$ for
$\dd\Q\otimes\dd s\otimes\dd x$-a.e. $(\omega,s,x)$ with
$s\le\theta$. Tonelli's theorem and the change of variables
$y=\Phi_s(x)$ yield
\begin{equation}
 \E^\Q\int_0^\theta\int_{\R^d}
 |\bar g_s(y)|^2J_s(\Psi_s(y))^{-1}\dd y\dd s=0.
 \label{eq:integrated-g-eulerian}
\end{equation}
The Jacobian factor is finite and strictly positive, so
$\bar g=0$ for $\dd\Q\otimes\dd s\otimes\dd x$-a.e.
$(\omega,s,x)$ with $s\le\theta$. This proves uniqueness.
\end{proof}

\begin{proof}[Proof of Theorem \ref{thm:main-bspde}]
Propositions~\ref{prop:density-stability} and
\ref{prop:uniform-parameter} construct the untruncated
characteristic BSDE solution and establish its spatial
regularity. Proposition~\ref{prop:diagonal-Y} gives the
continuous, observation-adapted Eulerian field $f$.
Proposition~\ref{prop:characteristic-equation} proves the
predictability of $g$ and the weak BSPDE identity directly
for these fields. Propositions~\ref{prop:weighted-Sobolev}
and \ref{prop:f-bound} verify the solution norm and boundedness,
and Proposition~\ref{prop:weighted-pointwise} gives
\eqref{eq:weighted-sup-2}. Finally,
Proposition~\ref{prop:local-time-continuity} gives the asserted
local continuity, and
Proposition~\ref{prop:bspde-uniqueness} proves uniqueness.
\end{proof}

\section{Concluding remarks}
\label{sec:conclusion}

We have studied the observation-adapted BSPDE directly under
a Brownian reference measure.   Its coefficients may have the stated spatial growth and may be
degenerate, while positive-order spatial derivatives and the terminal field satisfy the
explicit boundedness assumptions. The solution is unique among weighted Sobolev solutions with
bounded first component,  and  the polynomial weights can be chosen
independently of the control stopping level. 

The results established here will be applied in our forthcoming
work on nonlinear filtering with unbounded random coefficients
\cite{xu}. The BSPDE solution serves as an observation-adapted
backward test field for the stopped Zakai equation.  With suitably chosen random terminal data,
the resulting duality identity yields pathwise uniqueness of
the stopped Zakai equation in the admissible measure class.
This uniqueness is then transferred to the Kushner--Stratonovich
equation by unnormalization and removal of the control-energy
localization.

\end{document}